\documentclass[10pt, reqno]{amsart}
\usepackage{amssymb, hyperref}
\hypersetup{colorlinks=true,linkcolor=blue, citecolor=blue}
\usepackage{amsthm}
\usepackage{amsmath}
\usepackage{lineno}
\usepackage{graphicx}
\usepackage[all,2cell]{xy}
\usepackage{verbatim}
\usepackage{pst-node}
\usepackage{tikz-cd}

\immediate\write18{texcount -tex -sum  \jobname.tex > \jobname.wordcount.tex}
\usepackage{url}
\usepackage{setspace}
\numberwithin{equation}{subsection}
\newtheorem{theorem}{Theorem}[section]
\newtheorem{corollary}[theorem]{Corollary}
\newtheorem{definition}[theorem]{Definition}
\newtheorem{example}[theorem]{Example}
\newtheorem{lemma}[theorem]{Lemma}
\newtheorem{notation}[theorem]{Notation}
\newtheorem{proposition}[theorem]{Proposition}
\newtheorem{remark}[theorem]{Remark}

\newcommand{\eat}[1]{}

\renewcommand{\P}{\mathbb{P}}

\newcommand{\sslash}{\mathbin{/\mkern-6mu/}}

\newcommand{\RNum}[1]{\uppercase\expandafter{\romannumeral #1\relax}}

\title{Universal parabolic moduli over $\overline{M}_{_{g, n}}$}

\begin{document}
	\author{Jagadish Pine} 
		
	\thanks{email id : pine@cmi.ac.in}
	\date{11-04-2022}
	\maketitle
	
	{\small  Department of Mathematics, Chennai Mathematical Institute,  H1, Sipcot IT Park, Kelambakkam, Siruseri, Tamil Nadu 603103 \\
}
	\begin{abstract}
	In this article we will construct a universal moduli space of stable parabolic vector bundles over the moduli space of marked Deligne-Mumford stable curves $\overline{M}_{_{g, n}}$. The objects that appear over the boundary of  $\overline{M}_{_{g, n}}$ i.e., over singular curves will remain  vector bundles. The total space and the fibers over $\overline{M}_{_{g, n}}$ will have good singularities.
\end{abstract} \hspace{10pt}	


	\verbatiminput{\jobname.wordcount.tex}

	\section*{introduction}
	Let $(C, x^1, x^2, \cdots, x^n)$ be a marked semistable curve \eqref{markedsemistablecurve} and $\pi:C \rightarrow C'$ be the canonical contraction to the marked stable curve \eqref{modulistablecurves} $(C', x^1, x^2, \cdots, x^n)$, where the points $\{x^i: 1 \le i \le n\}$ are on the isomorphism locus of $\pi$. Given a vector bundle $E$ on the curve $C$ such that the restriction of $E$ to any rational component( $\cong \P^1$) of $C$ is strictly positive \eqref{positivebundle}, we define parabolic structures of $E$ at the points $\{x^i: 1 \le i \le n\}$ in the sense of Mehta-Seshadri. Following Nagaraj-Seshadri \cite{MR1687729}, we define such a parabolic bundle $E_*$ to be stable if $\pi_*E_*$ is a $p_2$-stable pure sheaf of dimension $1$ \eqref{stabilityofparGiebun} on  the marked stable curve $C'$ . We will denote such stable parabolic bundle $E_*$ as pair $(C, E_*)$ over the curve $C'$. Let $\overline{M}_{_{g, n}}$ be the moduli space of marked stable curves of genus $g \ge 2$ and number of marked points $n \ge 1$ \eqref{modulistablecurves}. Our aim in this article is to construct a universal moduli space $\overline{\mathfrak{U}}_{_{g, n, r }}= \overline{\mathfrak{U}}_{_{g, n}}(r, d, r^i_j, \alpha^i_j)$ over $\overline{M}_{_{g, n}}$. This moduli space $\overline{\mathfrak{U}}_{_{g, n, r }}$ parameterizes stable parabolic Gieseker pairs $(C, E_*)$ \eqref{Giesekerbundle} \eqref{stabilityofparGiebun}  with fixed numerical invariants rank $r$, degree $d$, quasi parabolic structures $r^i_j$, rational weights $\alpha^i_j$. Dirk Schl\"{u}ter in \cite{schluter2011universal} has shown that there exist a universal moduli space  $\overline{\mathcal{U}}_{_{g, n, r}}$ of parabolic pure sheaves over $\overline{M}_{_{g, n}}$. The fiber of the moduli space $\overline{\mathcal{U}}_{_{g, n, r}}$ over $[C'] \in \overline{M}_{_{g, n}}$ is the moduli space of $p_2$-stable \eqref{p_2stability} parabolic pure sheaves on $C'$ modulo $\text{Aut}(C', x^1, \cdots, x^n)$. Schl\"{u}ter follows the approach of Pandharipande \cite{MR1308406}. The novel features in the Gieseker type construction are: \\
	(1) There exist a proper map from the moduli space $\overline{\mathfrak{U}}_{_{g, n, r }}$ to the moduli space $\overline{\mathcal{U}}_{_{g, n, r}}$.\\
	(2) The objects in the degenerate fibers remain locally free.\\
	(3) The resulting moduli space $\overline{\mathfrak{U}}_{_{g, n, r }}$ has nicer singularities \eqref{mainthm2} \eqref{mainthm3} than $\overline{\mathcal{U}}_{_{g, n, r}}$.\\
	 In fact under conditions of semistable=stable \eqref{semistablemoduli}[ $2^{nd}$ paragraph], our moduli stack gives a birational smooth model of the one obtained in \cite{schluter2011universal}.\\
	
	The motivation to study the problem is as follows. Let $\text{G}$ be an almost simple, simply connected algebraic group over $\mathbb{C}$ and $\rho:\text{G} \rightarrow \text{GL}(k, \mathbb{C})$ be a faithful representation of $\text{G}$. Let $C$ be a smooth algebraic curve of genus $g \ge 2$ over $\mathbb{C}$. Then the upper half plane $\mathbb{H}$ is the universal cover and $\mathbb{H} \slash \pi =C$, where $\pi$ is certain Fuchsian group. The $(\pi, \text{G})$ bundles on $\mathbb{H}$ corresponds to parahoric torsors on $C$ \cite{MR3275653}. The construction of a parabolic Gieseker moduli space over $\overline{M}_{g,n}$ has an important application. Let $\mathcal{G}$ be a parahoric group scheme over the smooth projective curve $C$ with generic fibre $\text{G}$, where the parahoric structure is given at a set of $n$ marked points.  Then to construct degenerations of the moduli space $M(\mathcal{G})$ either when the  smooth curve degenerates to an irreducible nodal curve or to construct a universal moduli for $\mathcal{G}$-torsors over $\overline{M}_{g,n}$, one can follow the classical construction due to Ramanathan. This can be achieved by viewing the parahoric torsors as ``reduction of structure groups" of the parabolic Gieseker bundles. Then the stack of $\mathcal{G}$-torsors gets constructed from the stack of parabolic Gieseker bundles over $\overline{M}_{g,n}$ \cite{MR1420170}, \cite{balaji2022torsors}.\\
	
	There is no  torsion free analogue for the degeneration of $G$ bundles when a smooth curve degenerates to a nodal curve except for symplectic and orthogonal case due to Faltings \cite{MR1375622}. The only known way to study the degeneration problem of $G$ bundles is by using Gieseker type degeneration \cite{balaji2022torsors}. The following is a brief history of Gieseker type degeneration.\\ 
	
	 The first degeneration of this nature was constructed by Gieseker \cite{MR739786} for the rank $2$ and odd degree to prove the Newstead-Ramanan conjecture \cite[Theorem 1.1]{MR739786} by degeneration to an irreducible nodal curve. Later Nagaraj-Seshadri defined an equivalent notion of stability more in the spirit of Mumford stability to generalize the construction to arbitary rank $r$, degree  $d$ with g.c.d $(r, d)=1$ \cite{MR1687729} ; using this notion of stability, Schmitt constructed a universal moduli space of vector bundles over $\overline{M}_{_{g}}$ \cite{MR2106123} when the genus $g$ is greater than $1$.\\
	 
	 To state the main results of the article we will need the following definitions and notations. 

	\begin{definition}
		\begin{enumerate}
			\label{markedprestablecurves}
				\item A {\em marked prestable} curve $C$ is a reduced, connected,  projective curve whose singularities are nodal singularities with $n$ distinct nonsingular points on it.
				
				\label{markedsemistablecurve} 
			\item A {\em marked prestable curve} $C$ of genus $g \ge 2$ is called marked semistable if every non singular rational components ($\cong \P^1$) contains at least $2$ special points i.e., either marked points or singular points.
		\end{enumerate}
		
	\end{definition}

\begin{definition}
	Let $C$ and $D$ be two marked semistable curves and $E_*, F_*$ be two strictly positive $p_2$-stable parabolic vector bundle on $C$ and $D$ respectively. Then the pair $(C, E_*), (D, F_*)$ is Aut-equivalent if there exist a marked isomorphism $\phi:C \rightarrow D$ such that we have an isomorphism $E_* \cong \phi^* F_*$ of parabolic bundles \cite[definition 4.2.4]{schluter2011universal}.
\end{definition}

\begin{notation}
  Let $V_l$ be a vector space over complex numbers $\mathbb{C}$ and $r^i_{l_i+1} > r^i_{l_i} > \cdots > r^i_2$ be a decreasing sequence of positive integers for $1 \le i \le n$. Let $\text{Flag}(V_l, r^i_{l_i+1}, r^i_{l_i}, \cdots, r^i_2)$ be the flag variety of successive quotient spaces of $V_l$  of the form $V_l \rightarrow Q^i_{l_i+1} \rightarrow Q^i_{l_i} \rightarrow \cdots \rightarrow Q^i_2$, where $\text{dim}(Q^i_j)=r^i_j$ for $2 \le j \le (l_i+1)$. Then we denote by $\text{Fl}(V_l, \mathbb{C})=\prod_{i=1}^{n} \text{Flag}(V_l, r^i_{l_i+1}, r^i_{l_i}, \cdots, r^i_2)$.
\end{notation}

	The main result of this article is the following
	
	\begin{theorem}
		\label{mainthm1}
		There exist a projective variety $\overline{\mathfrak{U}}_{_{g, n, r }}$ over $\overline{M}_{_{g,n}}$ such that the fiber over a marked stable curve $C'$ parameterizes aut-equivalance classes of pairs $(C, E_*)$ where $C$ is a marked semi-stable curve whose fixed marked stable model is $C'$ and $E_*$ is a stable parabolic Gieseker bundle of fixed numerical type on $C$. Furthermore we have the following commutative diagram
		
		\begin{equation}
		\label{maindiagram}
			 \begin{tikzcd}
			\overline{\mathfrak{U}}_{_{g, n, r }} \arrow[rr, "\text{projective}" near start, "\text{birational}"', "\pi_*" near end] \arrow[dr, "\kappa_g"]
			& &\overline{\mathcal{U}}_{_{g, n, r}} \arrow[dl, "\eta"]\\%
			& \overline{M}_{_{g, n}}
			\end{tikzcd}
		\end{equation}
	
	\end{theorem}
	
	The following two theorems describes certain geometric properties of the total space and the fibers of the morphism $\kappa_g$ \eqref{maindiagram}.
	
	\begin{theorem}
		\label{mainthm2}
		The universal moduli $\overline{\mathfrak{U}}_{_{g, n, r }}$ is a normal projective variety with finite quotient singularity. The dimension of the moduli space is $3g-3+n+r^2(g-1)+1+\text{dim}\left(\text{Fl}(V_l, \mathbb{C})\right)$. \\
		
		The variety $\overline{\mathfrak{U}}_{_{g, n, r }}$ has a distinguished smooth open subvariety consisting of strictly stable bundles. 
	\end{theorem}
	
	\begin{theorem}
		\label{mainthm3}
		Let the marked stable curve $C'$ represent an element in   $\overline{M}_{_{g, n}}$ such that $\text{Aut}(C', x^1, \cdots , x^n)$ is trivial then $\kappa_g^{-1}([C'])$ has a singularity which is a product of analytic normal crossings.
	\end{theorem}

We will briefly sketch the main steps of the construction. As a first step we will prove that the above mentioned objects of our moduli problem i.e., pairs $(C, E_*)$ where $C$ is a marked semistable curve and $E_*$ is strictly positive $p_2$-stable bundle, form a bounded family. We will rigidify the moduli problem to show that the underlying vector bundles are in an open subvariety $\text{Y}_{_{g, n}}$ of a relative Hilbert scheme $\text{Hilb}^{p(t)}(\mathcal{X} \times \text{Gr}(V_l, r))$. The rigidification will induce a natural action of $\text{SL}(N) \times \text{SL}(V_l)$ on $\text{Y}_{_{g, n}}$. We will define a functor which is represented by a generalized flag variety $\mathfrak{F}_l$ over $\text{Y}_{_{g, n}}$. The action of $\text{SL}(N) \times \text{SL}(V_l)$ will lift to an action on $\mathfrak{F}_l$. We will show that there is a natural morphism from $\mathfrak{F}_l$ to the flag variety $\text{F}_l$  of the torsion free moduli problem. We will show that this morphism is proper. The morphism is defined by the pushforward $\pi_*$. The candidate for our universal moduli space is the GIT quotient $\mathfrak{F}_l \sslash \text{SL}(N) \times \text{SL}(V_l)$. One key element of a moduli space construction using GIT is the choice of a suitable group $\text{G}$ (here $\text{G}=\text{SL}(N) \times \text{SL}(V_l)$ )-linearized ample line bundle on the appropriate $\tt{Q}uot$ scheme (here $\mathfrak{F}_l$ ). Here we use in an essential way the ample linearization of the GIT construction in  \cite{schluter2011universal}. Indeed, the role that Schl\"{u}ter's paper plays on our paper is the exact counterpart of the role that the moduli space of torsion free sheaves plays on Nagaraj-Seshadri's paper \cite{MR1687729}. \\

The layout of the paper is as follows. In the section $\text{I}$ \eqref{preliminaries} we will give the  necessary definitions and review the universal parabolic torsion free moduli construction \cite{schluter2011universal}. In section $\text{II}$ \eqref{moduli problem} we will define the notion of stability of parabolic Gieseker bundle on a marked semistable curve. We will prove that these objects of our moduli problem form a bounded family and we define the parabolic Gieseker functor \eqref{parGiefunctor} over the Gieseker functor \eqref{Giefunctor}. In section $\text{III}$ \eqref{Some technical lemmas} we will show that the parabolic Gieseker functor is represented by a flag scheme $\mathfrak{F}_l$ and we will establish a morphism $\eta:\mathfrak{F}_l \rightarrow \text{F}_l$, where $\text{F}_l$ is the  flag variety \eqref{fiberprodprop} for the torsion free moduli problem. In section $\text{IV}$ \eqref{moduli construction} we will prove  the properness of the morphism $\eta$ and give the proof of the main theorem \eqref{mainthm1}. In section  $\text{V}$ \eqref{properties of the moduli space} we will prove \eqref{mainthm2} and \eqref{mainthm3}.

\renewcommand{\abstractname}{Acknowledgements}
\begin{abstract}
	I express my sincere gratitude to Prof. V. Balaji who suggested the problem to me and held numerous discussions related to the problem. I am extremely thankful to Prof. Sukhendu Mehrotra with whom I started learning the subject algebraic geometry and geometric invariant theory. Further I thank Dr. Sourav Das for many helpful discussions on moduli theory. I thank the referee for a careful reading of the manuscript and suggesting many changes which has improved the exposition of the article.
\end{abstract}
	
	\section{preliminaries}
	\label{preliminaries}
	In this section first we will give the necessary definitions. Then we will mention key features of \cite{schluter2011universal}. This is essential in our construction since the crucial polarization for our GIT construction will come from the moduli space constructed in \cite{schluter2011universal}.\\
	
	Throughout the section a curve will mean a marked curve with $n$ distinct nonsingular points on it. The standard notation is $C$ unless the markings need specifying. The arithmetic genus of the curve is $g$. 
	
\begin{definition} 

 A {\em pure sheaf} $\mathcal{E}$ of dimension $1$ on a marked prestable curve \eqref{markedsemistablecurve} $C$ is a coherent $\mathcal{O}_{C}$ module such that every non zero subsheaf $\mathcal{F}$ of $\mathcal{E}$ has dim(supp($\mathcal{F}))=1$. 

\end{definition}

Let $\mathcal{O}_{C}(1)$ be a polarization on $C$ and $a_i=deg(\mathcal{O}_{C}(1)_{|C_i}$), where $C_i$ be the irreducible components of $C$. Without loss of generality we also assume $\sum a_i=1$. Let $r_i=rank(\mathcal{E}_{|C_i})$. We define total rank of $\mathcal{E}$, $\text{totr}$$(\mathcal{E}):=\sum(r_i \cdot a_i)$. The degree of $\mathcal{E}$ is defined as, deg$(\mathcal{E}):=\chi(\mathcal{E})-\text{totr}(\mathcal{E})(1-g)$, where the Euler characteristic $\chi(\mathcal{E})= \tt{dim} \: H^0(\mathcal{E})-dim \: H^1(\mathcal{E})$. The Hilbert polynomial with respect to the polarization $\mathcal{O}_{C}(1)$ denoted by $H$ is $H(t):=\chi(\mathcal{E}(t))=\chi(\mathcal{E})+\text{totr}(\mathcal{E}) \cdot t$ by Riemann-Roch theorem.  \\

We now fix the $n$ marked points $\{x^1, x^2, \cdots, x^n\}$ in $C$.

\begin{definition}
	\label{parabolicstructure}
	 A quasi parabolic structure(QPS) on $\mathcal{E}$ at the points $x^i$ is defined as the following filtration of sheaves
	\begin{equation}
	\label{qps}
	\mathcal{E}=F^i_1 \mathcal{E}  \supseteq F^i_2 \mathcal{E} \supset F^i_3 \mathcal{E} \supset \cdots \supseteq F^i_{l_i+1} \mathcal{E}=\mathcal{E}(-x^i)
	\end{equation}
	
	A parabolic structure is a QPS with added weights
	\begin{equation}
	0 \le \alpha^i_1 < \alpha^i_2 < \cdots <\alpha^i_{l_i} <1
	\end{equation}
\end{definition}

\begin{notation}
	We will denote the Hilbert polynomial of $\mathcal{E}/F^i_j \mathcal{E}$ by $H^i_j$ with $\{1 \le i \le n\}$ and $\{2 \le j \le (l_i+1)\}$. It is easy to see that $H^i_j$ are constant integers. {\em Alternatively we will use the notation $\bf{r^i_j}$ for them}.
\end{notation}
 The above definition \eqref{qps} is due to Maruyama-Yokogawa \cite{MR1162674}. This is known to be equivalent to the classical notion of parabolic structure in \cite{MR575939} of the following form
  \begin{equation}
  	\mathcal{E}_{|x^i}=\mathcal{E}/\mathcal{E}(-x^i) \rightarrow \mathcal{E}/F^i_{l_i} \mathcal{E} \rightarrow \cdots \rightarrow \mathcal{E}/F^i_2 \mathcal{E}
  \end{equation}
  
  \begin{notation}
   Throughout this article a parabolic sheaf will be denoted by $\mathcal{E}_*$.
  \end{notation}

	 The parabolic degree is defined as:
	 \begin{center}
	 	 $\text{par deg}$($\mathcal{E}_*$):= $\text{deg}(\mathcal{E})+\sum_{1}^{n} \sum_{1}^{l_i} \alpha^i_j \cdot \text{dim}(F^i_j \mathcal{E}/F^i_{j+1} \mathcal{E})$
	 	 =$\text{deg}(\mathcal{E}$)+$\sum_{1}^{n} \sum_{1}^{l_i} \alpha^i_j \cdot (r^i_{j+1}-r^i_j)$, $\;$  $r^i_1=0$
	 	 \end{center}
	 
	 A parabolic sheaf $\mathcal{E}_*$ is defined to be slope stable(resp. slope semistable) if for any non zero proper saturated subsheaf $\mathcal F$ of $\mathcal E$ with the induced parabolic structure on $\mathcal{F}$ we have the inequality:
	 \begin{equation}
	 \label{slopestability}
	 	 \frac{\text {par deg}(\mathcal{F}_*)}{\sum a_i \cdot s_i} <(\le) \frac{\text{par deg}(\mathcal{E}_*)}{\sum a_i \cdot r_i}
	 \end{equation} 
	 where $s_i$ are the multirank of $\mathcal{F}$.\\
	 
	 \begin{example} \em ({A sheaf with different multirank}) Let $C$ be a reducible curve with two smooth irreducible components $C_1$ and $C_2$ meeting at a point $p=C_1 \cap C_2$.
	\tikzset{every picture/.style={line width=0.75pt}} 
	
	\tikzset{every picture/.style={line width=0.75pt}} 
	
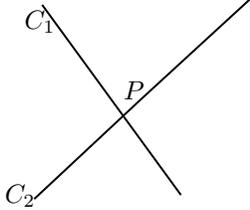
	\begin{tikzpicture}[x=0.75pt,y=0.75pt,yscale=-1,xscale=1]
	
	\draw    (670.33,91) -- (740.33,187) ;
	\draw    (666.33,189) -- (720.08,138.75) -- (775.33,88) ;
	
	\draw (660,92.4) node [anchor=north west][inner sep=0.75pt]    {$C_{1}$};
	\draw (650,180.4) node [anchor=north west][inner sep=0.75pt]    {$C_{2}$};
	\draw (710,140.4) node [anchor=south west][inner sep=0.75pt]    {$P$};

	\end{tikzpicture}
	
	Let $\mathcal{E}_1$, $\mathcal{E}_2$ be vector bundles of rank $r_1$, $r_2$ on the curves $C_1$, $C_2$ respectively. Let $\phi:(\mathcal{E}_1)_P \rightarrow (\mathcal{E}_2)_P $. We consider the map $\psi: \mathcal{E}_1 \oplus \mathcal{E}_2 \rightarrow (\mathcal{E}_2)_P $ defined by $\psi(s_1, s_2)=\phi(s_1(P))-s_2(P)$, where $s_1(P), s_2(P)$ are the restriction of the sections $s_1, s_2$ on the fiber $(\mathcal{E}_1)_P, (\mathcal{E}_2)_P$ respectively. Let $\text{ker}(\psi)$ be $\mathcal{E}$. Then $\mathcal{E}$ is a pure sheaf on the curve $C$. If we choose $r_1 \ne r_2$ then the multiranks of $\mathcal{E}$ are different.  
	 		 \end{example}
	 
\begin{definition}
	Following D. Schl\"{u}ter \cite[definition 4.3.11]{schluter2011universal} we call a parabolic sheaf $\mathcal{E}_*$ to be $p_2$-stable(resp. $p_2$- semistable) if  for any non zero proper saturated subsheaf $\mathcal F$ of $\mathcal E$ with the induced parabolic structure on $\mathcal{F}$, multirank $s_i$, and Hilbert polynomial  $H(\mathcal{F}/F^i_j \mathcal{F})= s^i_j$, we have the inequality:
	\footnotesize	\begin{equation}
	\label{p_2stability}
	H(\mathcal{E}, y)	\cdot \left(H(\mathcal{F}, x)-\frac{1}{n}\left(\sum_{1}^{n}\sum_{2}^{l_{i+1}}\epsilon^i_j \cdot s^i_j\right)\right)  <(\le )	H(\mathcal{F}, y)	\cdot \left(H(\mathcal{E}, x)-\frac{1}{n}\left(\sum_{1}^{n}\sum_{2}^{l_{i+1}}\epsilon^i_j \cdot r^i_j\right)\right)
	\end{equation} \normalsize
	where $\epsilon^i_j:=\alpha^i_j-\alpha^i_{j-1}$, $2 \le j \le (l_i+1)$ and $\alpha^i_{l_i+1}=1$. The above inequality of polynomials in two variables $x, y$ is according to the lexicographic ordering in $\mathbb{Q}[x, y]$ \eqref{lexicographic}.
\end{definition}	  
	
\begin{remark}
	The notion of $p_2$- stability that has been used to construct the universal moduli space in \cite{schluter2011universal} is a modification of the notion of slope stability \eqref{slopestability}. In the definition \cite[definition 4.3.11]{schluter2011universal} the Hilbert polynomials $H^i_j(\mathcal{E}, x)$ and $H^i_j(\mathcal{E}, y)$ will become constants for curves. The constant is $\text{dim}(\mathcal{E}/F^i_j \mathcal{E})=r^i_j$. Also for a subsheaf $\mathcal{F}$ of $\mathcal{E}$, the Hilbert polynomial $p(\mathcal{F}/F^i_j (\mathcal{F}), y)$ is the constant $s^i_j$. Therefore putting all these quantities together the definition \cite[definition 4.3.11]{schluter2011universal} will take the form of \eqref{p_2stability}.
\end{remark}

	\begin{remark}
		In the above definition of $p_2$-stability, the product of polynomials that appears on both side of \eqref{p_2stability} are of the form $(a_1 y+b_1)(a_2x+b_2)$. In the next paragraph we will define lexicographic ordering for this set of polynomials.  
	\end{remark}
	
\begin{definition}
\label{lexicographic}
	Given two polymials $f=a_1xy+b_1y+c_1x+d_1$ and $g=a_2xy+b_2y+c_2x+d_2$ in $\mathbb{Q}[x,y]$, we will define $f \le g $ if $a_1 <a_2$ or $a_1 = a_2$ and $b_1 <b_2$ or $a_1 = a_2 , b_1 =b_2$ and $c_1 < c_2$ or $a_1 = a_2 , b_1 =b_2 , c_1 = c_2$ and $d_1 \le d_2$.
\end{definition}

For few more details see Appendix \ref{appendixb}.

\subsection{Universal moduli of parabolic sheaves} We will briefly give an outline of the universal moduli of semi stable parabolic pure sheaves as in \cite{schluter2011universal}.\\

Objects of the moduli space are $p_2$-stable(semistable) \eqref{p_2stability} pure sheaves $\mathcal{E}$ of uniform rank $r$ i.e., $r_i=r \; \forall i$ on marked stable curves $C$ \eqref{stablecurve}. We will fix once and for all rank $r$ , deg $d$, quasi parabolic structures $r^i_j$ and weights $\alpha^i_j \in \mathbb{Q}$. This data is called numerical type of this moduli problem. \\

By (semi)stable locus, (semi)stable points of a variety $X$ under the action of a reductive group $G$ we mean it in the sense of GIT.

\begin{notation}
	\label{S_{g, n} notation}
	Let $\mu:\mathcal{X} \rightarrow {\tt S_{_{g,n}}}$ \eqref{localunifamily} be the local universal family of marked stable curves with natural sections $\sigma_i:{\tt S_{_{g,n}}} \rightarrow \mathcal{X}, \; 1 \le i \le n$ corresponding to the marked points. Let $D^i$ be their associated divisors. The base ${\tt S_{_{g,n}}}$ is the Quot scheme for the moduli problem of $\overline{M}_{_{g, n}}$\eqref{localunifamily}.  Let $\mathcal{O}_{\mathcal{X}}(1)$ be the relative ample line bundle $\left(\omega_{\mathcal{X}/{\tt S_{_{g,n}}}}(\sum \sigma_i)\right)^{\otimes \rho}$.
\end{notation}
\begin{remark}
In several places for notational reasons we will denote $\bf{{\tt S_{_{g,n}}}}$ by $\tt{S}$.
\end{remark}
 By the boundedness of the family of $p_2$-stable sheaves on marked stable curves of fixed numerical type \eqref{boundedness} there exist $l_0$ such that $\forall \;l \ge l_0$ and $\forall \;s \in {\tt S_{_{g,n}}}$ we have:\\

(1) $\text{H}^1(\mathcal{X}_s, \mathcal{E}(l))=0$ and (2) $\text{H}^0(\mathcal{X}_s, \mathcal{E}(l)) \otimes \mathcal{O}_{\mathcal{X}_s} \rightarrow \mathcal{E}(l)$ is  surjective.\\
where $\mathcal{E}$ is a $p_2$-stable sheaf on $\mathcal{X}_s$.\\
 
 Since degree of $\mathcal{E}(l)$ is independent of $s \in {\tt S_{_{g,n}}}$,  from the cohomology vanishing $\text{H}^1(\mathcal{X}_s, \mathcal{E}(l))=0$ and the Riemann-Roch theorem, it follows $\text{dim}\ \text{H}^0(\mathcal{X}_s, \mathcal{E}(l))$ is independent of $s \in {\tt S_{_{g,n}}}$. We will rigidify the moduli problem by fixing an isomorphism $\text{V}_l \cong \text{H}^0(\mathcal{E}(l))$. 
  We consider the relative Quot scheme over ${\tt S_{_{g,n}}}$:
 \begin{equation}
 \label{relquotscheme}
 	\text{Quot}^{\mathcal{O}_{\mathcal{X}}(1)}_{\mathcal{X}/{\tt S_{_{g,n}}}}(\text{V}_l \otimes \mathcal{O}_{\mathcal{X}}, \text{H}) 
 \end{equation}
  where $\tt{H}$ is the Hilbert polynomial $\text{H}(\mathcal{E}(l), t)=\chi(\mathcal{E}(l))+r \cdot t=d+r \cdot l+r \cdot (1-g)+r \cdot t$. We will denote the Quot scheme \eqref{relquotscheme} by $\text{Q}_g(\mu, V_l, H)$ where $\mu:\mathcal{X} \rightarrow \tt S_{_{g,n}}$ is the structure map. The points of the Quot scheme $\text{Q}_g(\mu, V_l, H)$ are coherent $\mathcal{O}_{\mathcal{X}}$ module $\mathcal{F}$ which are flat over $S_{_{g,n}}$ with a surjection $V_l \otimes \mathcal{O}_{\mathcal{X}} \rightarrow \mathcal{F} \rightarrow 0$ and for $s \in S_{_{g,n}}$ the fiber $\mathcal{F}_s$ over $\mathcal{X}_s$ has Hilbert polynomial $H$ with respect to the ample line bundle $\mathcal{O}_{\mathcal{X}_s}(1)$.
  \begin{remark}
  	Note that from the above $(2)$, the $p_2$-stable sheaves $\mathcal{E}(l)$ are points of $\text{Q}_g(\mu, V_l, H)$.
  \end{remark}
 The rigidification will induce an action of $\text{SL}(V_l)$ on $\text{Q}_g(\mu, V_l, \text{H})$. There is a natural group action of $\text{SL}(N) \times \text{SL}(V_l) $ on $\text{Q}_g(\mu, V_l, H)$ where $\tt{SL}(N)$ action comes from it's action on $\mathcal{X}$ \eqref{modulistablecurves}. The universal quotient $\mathcal{U}$ fits in the universal short exact sequence
 \begin{equation}
 \label{unitorsionses}
 	0 \rightarrow \mathcal{K} \rightarrow \text{V}_l \otimes \mathcal{O}_{\mathcal{X} \times_{\text{S}} \text{Q}_g(\mu, V_l, H)} \rightarrow \mathcal{U} \rightarrow 0
 \end{equation}
 where the kernel $\mathcal{K}$ is called the universal subsheaf.\\
 
 Let $\text{F}_l$ be the flag variety over $\text{Q}_g(\mu, V_l, H)$ which parameterizes parabolic sheaves  $\mathcal{E}_*$ of fixed numerical type $(\tt{H}, r^i_j, \alpha^i_j)$ such that $\mathcal{E} \in \tt{Q}_g(\mu, V_l, H)$. More precisely for $s \in S_{_{g,n}}$ a closed point of the fiber $(\text{F}_l)_s$ is given by a quotient $V_l \otimes \mathcal{O}_{\mathcal{X}_s} \rightarrow \mathcal{E}$ in the fiber $(\text{Q}_g(\mu, V_l, H))_s$ along with a sequence of quotients $\mathcal{E} \rightarrow Q^i_{l_i+1} \rightarrow Q^i_{l_i} \rightarrow \cdots \rightarrow Q^i_2$, where $Q^i_j$ are skyscraper sheaves supported at the parabolic point $x^i$ for $i=1, 2, \cdots , n$ and $\text{dim}(Q^i_j)=r^i_j$ for $j=2, 3, \cdots , l_{i+1}$.  This has been constructed  inductively by constructing successive Quot schemes using the universal quotient sheaf $\mathcal{U}$ in \cite[4.6]{schluter2011universal}. The flag variety $\text{F}_l$ represents the following functor:
 \begin{equation}
 \label{flagfunctor}
 \underline {\text{F}_l}:\text{Sch}/{\tt S_{_{g,n}}} \rightarrow \text{Sets}
 \end{equation} 
 for a ${\tt S_{_{g,n}}}$ scheme $T$, we define $\underline{\text{F}_l}(T)$ to be the set

 \begin{center}
 	\{parabolic filtrations of $T$-flat sheaves $\mathcal{E}_{T}$ of the form \eqref{relparfiltration} which has a quotient representation $\text{V}_l \otimes \mathcal{O}_{\mathcal{X}_T} \rightarrow \mathcal{E}_{T}$ such that for all $t \in T$, $\mathcal{E}_t$ has Hilbert polynomial $\tt{H}$ and $(\mathcal{E}/F^i_j \mathcal{E})_t$ has dimension $\tt{r}^i_j$ \}
 \end{center}

 The natural action of $\text{SL}(N) \times \text{SL}(V_l)$ on $\text{Q}_g(\mu, V_l, H)$ will lift to an action on $\text{F}_l$. The GIT quotient ${\displaystyle \text{F}_l \sslash \text{SL}(V_l) \times \text{SL(N)}}$ is the candidiate for the universal moduli space. The moduli functor associates to every scheme $T$ the set 
 
 \begin{center}
 	\{equivalance classes of flat family of $p_2$-stable\footnote{The moduli space in \cite{schluter2011universal} has been constructed for $p_2$-semistable sheaves. A notion of $S$-equivalance has been defined for $p_2$ semistable sheaves.} parabolic sheaves $\mathcal{E}_*$ \eqref{flatfamilyofparabolicsheaves} on the family of marked stable curves $\mu_{T}:C_T \rightarrow T$ such that $\mathcal{E}_s$ has uniform rank $r$, degree $d$ and $(\mathcal{E}/F^i_j \mathcal{E})_s$ has dimension  $r^i_j$  \}
 \end{center}

 Two flat families of $p_2$-stable sheaves $\mathcal{E}_*$ and $\mathcal{E'}_*$ on $C_T/T$ and $C'_T/T$ are equivalent if there exist a $T$ isomorphism $\psi:C_T \rightarrow C'_T$ and a line bundle $\mathcal{L}$ on $T$ such that we have a parabolic isomorphism $ \mathcal{E}_* \cong \psi^* \mathcal{E'}_* \otimes \mu_{T}^* \mathcal{L}$. Let 
 \begin{equation}
 \label{uniformrank}
 	\text{Q}^r_g(\mu, V_l, H) \hookrightarrow \text{Q}_g(\mu, V_l, H)
 \end{equation}
  be the closed subscheme of uniform rank $r$ \cite[Lemma 8.1.1]{MR1308406}. The uniform rank Quot scheme $\text{Q}^r_g(\mu, V_l, H)$ is invariant under the action of $ \text{SL(N)}\times \text{SL}(V_l)$. By an abuse of notation $\text{F}_l$ continues to denote the base change of $\text{F}_l$ under the morphism \eqref{uniformrank}. \\
  
  The flag variety $\text{F}_l$ can be embedded as a closed subscheme inside the product of relative Grassmannians over ${\tt S_{_{g,n}}}$  \cite[4.41]{schluter2011universal} 
  \begin{equation}
  \label{rel embedding of flag}
  \text{F}_l \hookrightarrow \text{Gr} \times_S (\text{Gr}^1_{l_1+1} \times_S \cdots \times_S \text{Gr}^1_2) \times_S \cdots (\text{Gr}^n_{l_n+1} \times_S \cdots \times_S \text{Gr}^n_2)
  \end{equation}
  where $\text{Gr}:=\text{Gr}_S(V_l \otimes \mu_* \mathcal{O}_{\mathcal{X}}(k), H(k))$ and $\text{Gr}^i_j:=\text{Gr}_S(V_l \otimes \mu_* \mathcal{O}_{\mathcal{X}}(k), r^i_j(k))$. Here the pushforward of the relative ample line bundle $\mathcal{O}_{\mathcal{X}}(k)$ has been defined with respect to the structure map $\mu:\mathcal{X} \rightarrow \tt{S}_{_{g, n}}$ \eqref{S_{g, n} notation}. The embedding is $\text{SL}(V_l)$ equivariant.\\
  
  We will get the following $\mathbb{Q}$-divisor on the product of Grassmannians on RHS of \eqref{rel embedding of flag}
  \begin{equation}
  	\label{linearizationforSL(V_l)}
  	\tt{L}_{\beta, \beta^i_j}:=\mathcal{O}_{Gr}(\beta) \otimes_{i=1}^{n} \otimes_{j=2}^{l_i+1}\mathcal{O}_{Gr^i_j}(\beta^i_j)
  \end{equation}
 where $\beta, \: \beta^i_j $ are rational numbers. For $a \gg 0$,  ${\tt{L}_{\beta, \beta^i_j}^{\otimes a}}_{\big | \text{F}_l}$ is a relative very ample linearization for the $\text{SL}(V_l)$  action on $\text{F}_l$.
 \begin{remark}
  	It is enough to study the GIT problem of $\text{SL}(V_l)$ action on a fiber $(\text{F}_l)_s$ for $s \in S_{_{g, n}}$ by using \cite[Lemma 1.13]{MR1307297}.
  \end{remark}

\begin{notation}
	A point in the image of the map in \eqref{rel embedding of flag} over a point $s \in S_{_{g, n}}$ will correspond to a parabolic sheaf $\mathcal{E}_*$ together with a morphism $\gamma: \text{V}_l \rightarrow \text{H}^0(\mathcal{E}(l))$ which is induced by the quotient map $\text{V}_l \otimes \mathcal{O}_{\mathcal{X}_s} \rightarrow \mathcal{E}(l)$. 
\end{notation}
For a suitable choice of linearization weights $\beta, \beta^i_j$ \cite[4.49]{schluter2011universal} we will get the following lemma
  \begin{lemma}
  	\label{GITforSL(V_l)}
  	There exist $l_0 \in \mathbb{N}$ such that  for all $l \ge l_0$ there exist $K(l)$ with the property that for all $ k \ge K(l)$, a point $(\mathcal{E}_*, \gamma)$ in the flag variety $(\text{F}_l)_s$ is $\text{SL}(V_l)$ stable with respect to the linearization $ {\tt{L}_{\beta, \beta^i_j}^{\otimes a}}_{\big | \text{F}_l}$ if and only if $ \mathcal{E}_*$ is a $p_2$-stable sheaf of uniform rank $r$ and $\gamma:\text{V}_l \rightarrow \text{H}^0(\mathcal{E}(l))$ is an isomorphism. 
  \end{lemma}

 \begin{proof}
 	For a proof see \cite[Theorem 4.8.1]{schluter2011universal}.
 \end{proof}

\begin{remark}
	Notice that the linearization $ {\tt{L}_{\beta, \beta^i_j}^{\otimes a}}_{\big | \text{F}_l}$ depends on the integer $k$ since it is induced from the embedding \eqref{rel embedding of flag} which depends on $k$.
\end{remark}

 We will denote the product on the R.H.S \eqref{rel embedding of flag} by ${\bf{{Gr}_{S}(l, k)}}$. Let $\text{ \bf{{Gr}(l, k)}}$ be the fiber of $\tt{\bf{{Gr}_S(l, k)}}$ at a closed point $s$ in ${\tt S_{_{g,n}}}$. Then
 \begin{center}
  $\tt{ \bf{{Gr}(l, k)}}=Gr(V_l \otimes \text{H}^0(\mathcal{O}_{\mathcal{X}_s}(k), H(k)) \times_{\mathbb{C}} \prod_{i=1}^{n} \prod_{j=2}^{l_i+1} Gr(V_l \otimes \text{H}^0( \mathcal{O}_{\mathcal{X}_s}(k)), r^i_j(k))$.
 \end{center} 

 The flag variety $\text{F}_l$ can be embedded as a locally closed subscheme \cite[p. 139]{schluter2011universal}
 \begin{equation}
 \label{totalspaceF_lembedding}
 \text{F}_l \hookrightarrow {\tt S_{_{g,n}}} \times_{\mathbb{C}} \tt{ \bf{{Gr}(l, k)}} \hookrightarrow \overline{{\tt S_{_{g,n}}}} \times_{\mathbb{C}} \tt{ \bf{{Gr}(l, k)}}
 \end{equation}
 
 There exist a natural action of $\tt{SL}(N)$ on $\overline{{\tt S_{_{g,n}}}}$ and a linearization $\tt{L}_{m, m'}$ for the $\tt{SL}(N)$ action on $\overline{{\tt S_{_{g,n}}}}$ with nice properties \cite[ p. 47]{schluter2011universal}, \cite[pp. 23-24]{MR2431236}, \eqref{GITtheoremmodcurves}. We will write $\mathcal{M}=\tt{L}_{m, m'}$. We have mentioned the linearization $\tt{L}_{\beta, \beta^i_j}$ for the $\text{SL}(V_l)$ action on $\tt{ \bf{{Gr}(l, k)}}$. The embedding \eqref{totalspaceF_lembedding} is an $\text{SL}(N) \times \text{SL}(V_l)$- equivariant embedding. The following $\mathbb{Q}$-divisor 
  \begin{equation}
 \label{linearizationproductaction}
 	\tt{L}_{\beta, \beta^i_j} \boxtimes \mathcal{M}^{\otimes b}
 \end{equation}
 on $\overline{{\tt S_{_{g,n}}}} \times_{\mathbb{C}} \tt{ \bf{{Gr}(l, k)}}$ gives a linearization for the action of $\text{SL}(N) \times \text{SL}(V_l)$ on $\tt{ \bf{{Gr}(l, k)}}$ and therefore on $\text{F}_l$ by pullback. We will denote this linearization by $\overline{L}$.\\
 
 The semistable locus of $\text{F}_l$ is closed in the semistable locus of  $\overline{{\tt S_{_{g,n}}}} \times_{\mathbb{C}} \tt{ \bf{{Gr}(l, k)}}$ for the $\text{SL}(N) \times \text{SL}(V_l)$ action with respect the linearization $\overline{L}$ for $b \gg 0$ \cite[8.2, 30]{MR1308406}, \cite[p. 140]{schluter2011universal}. Therefore the GIT quotient ${\displaystyle \text{F}_l \sslash \text{SL}(N) \times \text{SL}(V_l)}$ exists as a projective variety which is denoted by $\overline{\mathcal{U}}_{_{g, n, r}}$. The following proposition together with \eqref{GITforSL(V_l)} gives the moduli theoretic interpretation of $\text{SL}(N) \times \text{SL}(V_l)$ stable points of $\text{F}_l$ with respect to the linearization $\overline{L}$ for $b \gg 0$.
 \begin{proposition}
 	\label{tfmoduliinterpretation}
 	A point of $\text{F}_l$ is stable for the $\text{SL}(V_l)$ action with respect to the linearization $ {\text{L}_{\beta, \beta^i_j}^{\otimes a}}_{\big | \text{F}_l}$ if and only if the point is GIT stable for the $\text{SL}(N) \times \text{SL}(V_l)$ action under the linearization $\overline{L}|_{\text{F}_l}$.
 	 \end{proposition}
  \begin{proof}
  	For a proof we refer to \cite[Proposition 5.1.2]{schluter2011universal}.
  \end{proof}

	\section{moduli problem}
	\label{moduli problem}
	The aim of this section is to state the moduli problem and then define a representable functor which is represented by a flag variety that parameterizes all parabolic structures of fixed numerical type on semi-stable curves. This functor is built together with a morphism to the Gieseker functor \eqref{Giefunctor}.

	\subsection{Parabolic Gieseker bundles} \label{basicdefGieseker}
	\begin{definition}
		\label{markedGiecurve}
		Let $C$ be a marked semistable curve of genus $g \ge 2$ \eqref{markedsemistablecurve} such that $\pi:C \rightarrow C'$ is the collapsing morphism and $C'$ is the fixed marked stable model. Then $C$ is called a {\em marked Gieseker curve} if $\pi^* \omega_{C'} \cong \omega_C$, where $\omega_{C'}$ and $\omega_C$ are the dualizing sheaves of $C'$ and $C$ respectively.
	\end{definition}
	Let $z^1, z^2, \cdots , z^c$ are the nodes of $C'$ such that $\pi^{-1}(z^j)=R^j$ is a chain of projective lines i.e., $R^j=\cup R^j_{\iota}$ such that $R^j_{\iota} \cong \P^1$, $R^j_{\iota} \cap R^j_{\iota'}$ is a singleton set if $|{\iota}-{\iota'}|=1$, otherwise empty and $R^j$ meets other components of $C$ at exactly two points $p^j_1, p^j_2$.\\

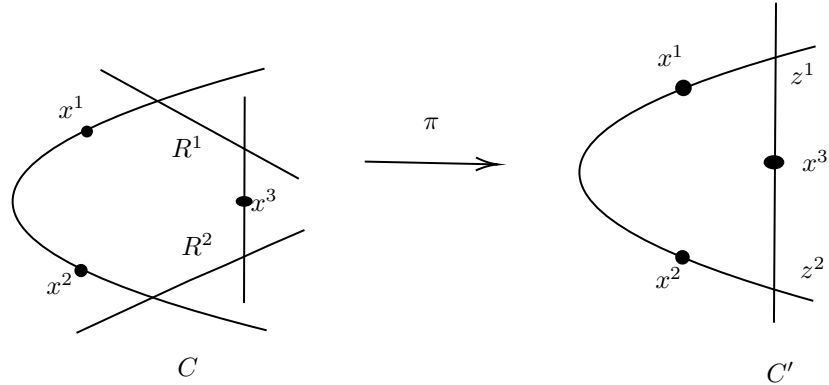
\tikzset{every picture/.style={line width=0.75pt}} 

\begin{tikzpicture}[x=0.75pt,y=0.75pt,yscale=-1,xscale=1]

\draw   (151.39,83.91) .. controls (-18.37,129.43) and (-18.01,173.46) .. (152.45,216.01) ;
\draw    (68.67,84.5) -- (168.67,139.5) ;
\draw    (56.57,217.3) -- (113.5,190.89) -- (131.9,182.98) -- (154.31,172.83) -- (165.83,167.85) -- (171.53,165.38) ;
\draw    (141.39,98.02) -- (141.3,125.69) -- (141.28,131.85) -- (141.26,138.89) -- (141.24,145.05) -- (141.14,174.97) -- (141.11,181.42) -- (141.04,202.25) ;
\draw  [fill={rgb, 255:red, 0; green, 0; blue, 0 }  ,fill opacity=1 ] (137.34,150.94) .. controls (137.33,149.73) and (139.02,148.74) .. (141.1,148.74) .. controls (143.18,148.75) and (144.88,149.73) .. (144.89,150.95) .. controls (144.89,152.16) and (143.21,153.15) .. (141.12,153.14) .. controls (139.04,153.14) and (137.35,152.16) .. (137.34,150.94) -- cycle ;
\draw    (201.95,130.84) -- (266.92,132.23) ;
\draw [shift={(268.92,132.28)}, rotate = 181.22] [color={rgb, 255:red, 0; green, 0; blue, 0 }  ][line width=0.75]    (10.93,-3.29) .. controls (6.95,-1.4) and (3.31,-0.3) .. (0,0) .. controls (3.31,0.3) and (6.95,1.4) .. (10.93,3.29)   ;
\draw   (429.33,72.67) .. controls (270.87,114.49) and (270.48,157.32) .. (428.14,201.13) ;
\draw    (409.35,50.67) -- (408.55,174.75) -- (408.31,212.17) ;
\draw  [fill={rgb, 255:red, 0; green, 0; blue, 0 }  ,fill opacity=1 ] (403.6,131.18) .. controls (403.6,129.48) and (405.71,128.11) .. (408.32,128.11) .. controls (410.93,128.12) and (413.04,129.5) .. (413.03,131.2) .. controls (413.03,132.9) and (410.92,134.27) .. (408.31,134.27) .. controls (405.7,134.26) and (403.59,132.88) .. (403.6,131.18) -- cycle ;
\draw  [fill={rgb, 255:red, 0; green, 0; blue, 0 }  ,fill opacity=1 ] (56,185.83) .. controls (56,184.27) and (57.27,183) .. (58.83,183) .. controls (60.4,183) and (61.67,184.27) .. (61.67,185.83) .. controls (61.67,187.4) and (60.4,188.67) .. (58.83,188.67) .. controls (57.27,188.67) and (56,187.4) .. (56,185.83) -- cycle ;
\draw  [fill={rgb, 255:red, 0; green, 0; blue, 0 }  ,fill opacity=1 ] (64.33,115.67) .. controls (64.33,114.29) and (63.21,113.17) .. (61.83,113.17) .. controls (60.45,113.17) and (59.33,114.29) .. (59.33,115.67) .. controls (59.33,117.05) and (60.45,118.17) .. (61.83,118.17) .. controls (63.21,118.17) and (64.33,117.05) .. (64.33,115.67) -- cycle ;
\draw  [fill={rgb, 255:red, 0; green, 0; blue, 0 }  ,fill opacity=1 ] (359,179.17) .. controls (359,177.42) and (360.42,176) .. (362.17,176) .. controls (363.92,176) and (365.33,177.42) .. (365.33,179.17) .. controls (365.33,180.92) and (363.92,182.33) .. (362.17,182.33) .. controls (360.42,182.33) and (359,180.92) .. (359,179.17) -- cycle ;
\draw  [fill={rgb, 255:red, 0; green, 0; blue, 0 }  ,fill opacity=1 ] (359,93.67) .. controls (359,91.64) and (360.64,90) .. (362.67,90) .. controls (364.69,90) and (366.33,91.64) .. (366.33,93.67) .. controls (366.33,95.69) and (364.69,97.33) .. (362.67,97.33) .. controls (360.64,97.33) and (359,95.69) .. (359,93.67) -- cycle ;

\draw (106,228.4) node [anchor=north west][inner sep=0.75pt]    {$C$};
\draw (230,107.4) node [anchor=north west][inner sep=0.75pt]    {$\pi $};
\draw (403,230.4) node [anchor=north west][inner sep=0.75pt]    {$C'$};
\draw (415,79.4) node [anchor=north west][inner sep=0.75pt]    {$z^{1}$};
\draw (420,177.4) node [anchor=north west][inner sep=0.75pt]    {$z^{2}$};
\draw (103,116.4) node [anchor=north west][inner sep=0.75pt]    {$R^{1}$};
\draw (108,165.4) node [anchor=north west][inner sep=0.75pt]    {$R^{2}$};
\draw (46,96.4) node [anchor=north west][inner sep=0.75pt]    {$x^{1}$};
\draw (40,185.4) node [anchor=north west][inner sep=0.75pt]    {$x^{2}$};
\draw (143,143.4) node [anchor=north west][inner sep=0.75pt]    {$x^{3}$};
\draw (421,123.4) node [anchor=north west][inner sep=0.75pt]    {$x^{3}$};
\draw (348,70.4) node [anchor=north west][inner sep=0.75pt]    {$x^{1}$};
\draw (347,182.4) node [anchor=north west][inner sep=0.75pt]    {$x^{2}$};

\end{tikzpicture}

	We fix the notation $E$ to denote a vector bundle on $C$, where $C$ will denote a marked Gieseker curve. The rank and degree of $E$ are denoted by $r$ and $d$ respectively. Then $E|_{R^j_{\iota}} \cong \oplus_{i=1}^{r} \mathcal{O}_{{\P}^1}(a^j_{\iota i}) $.
	
		\begin{definition}
			\label{positivebundle}
		The bundle $E$ is called positive if  $a^j_{\iota i} \ge 0$ for all $j, \iota, i$. The bundle $E$ is called strictly positive if $E$ is positive and for all $j, \iota$ there exist $i$ such that $ a^j_{\iota i} >0$. The bundle $E$ is strictly standard if it is strictly positive and $0 \le a^j_{\iota i} \le 1$.
	\end{definition}

	\begin{definition}
		\label{Giesekerbundle}
		A vector bundle $E$ on $C$ is called a Gieseker vector bundle if $E$ is strictly positive and $\pi_* E$ is a pure sheaf of dimension $1$ on $C'$.
	\end{definition}
	 
	Let $\{x^1, x^2, \cdots , x^n\}$ are the marked points on $C$. The marked points $\{x^1, x^2, \cdots, x^n\}$ are not on the chain of projective lines $R^j$ which are contracted to get a canonical marked stable curve.\\
	
	 A parabolic structure on $E$ at the points $x^i$ is a quasi parabolic structure(QPS) which is a decreasing filtration of sheaves
	\begin{equation}
	\label{qpsforGiesekerbundle}
		E=F^i_1 E  \supseteq F^i_2 E \supset F^i_3 E \supset \cdots \supseteq F^i_{l_i+1} E=E(-x^i)
	\end{equation}
	together with an increasing sequence of weights
	\begin{equation}
	0 \le \alpha^i_1 < \alpha^i_2 < \cdots <\alpha^i_{l_i} <1
	\end{equation}
	
 The map $\pi$ gives a canonical isomorphism $C-\cup R^j \cong C'-\{z^j:1 \le j \le c\}$. Therefore it is natural to use the same notation $x^i$ for the points $\pi(x^i)$ on $C'$. We will fix the line bundle $\left(\omega_{C}(\sum x^i)\right)^{\otimes \rho}$ on $C$ which is isomorphic with the pullback $\pi^{*}\left(\omega_{C'}(\sum x^i)\right)^{\otimes \rho}$. We will denote the line bundle $\left(\omega_{C}(\sum x^i)\right)^{\otimes \rho}$ by $\mathcal{O}_{C}(1)$.\\

Once and for all we will fix numerical type (a set of invariants) for our moduli problem---rank$(E)=r$, degree of $E=d$, Hilbert polynomials of $E/F^i_j E$ with respect to $\mathcal{O}_{C}(1)$ is $r^i_j$ which is the constant dim$_{\mathbb{C}} E/F^i_j E$ and weights $\alpha^i_j \in \mathbb{Q}$. \\

The parabolic Gieseker bundle will be denoted by $(E_*)$. Since $\pi_*$ is a left exact functor we have a filtration of sheaves on $C'$:
\begin{equation}
	\label{pushforwardofparabolicGie}
		\pi_* E=\pi_* F^i_1 E  \supseteq \pi_* F^i_2 E  \supset \cdots \supseteq \pi_* F^i_{l_i+1} E=\pi_*\left(E(-x^i)\right) \cong (\pi_* E)(-x^i)
\end{equation}
 
\begin{lemma}
	The parabolic sheaves $E_*$ on $C$ and $(\pi_*E)_*$ on $C'$ have the same numerical type $(r, d, r^i_j, \alpha^i_j)$.
\end{lemma} 
\begin{proof}
	It is clear that ranks will remain same under pushforward $\pi_*$. Since genus$(C)=$genus$(C')$ and $\text{H}^i(C, E) \cong \text{H}^i(C', \pi_* E)$ for $i \ge 0$ \cite[Proposition $3$]{MR1687729}, by using Riemann-Roch theorem we will get $\text{deg}(E)=\text{deg}(\pi_*E)=d$. \\
	 
 Let $U^i$ be a neighbourhood at $x^i$ such that restriction of $\pi$ induces an isomorphism $\pi:V^i=\pi^{-1}(U^i) \cong U^i$. Since both $E/F^i_j E$ and $\pi_* E/\pi_* F^i_j E$ are supported at points $x^i$ we will get 
 \begin{equation}
 \label{numericaltypeiso}
  E/F^i_j E \cong E|_{V^i}/F^i_j E|_{V^i} \cong \pi_*(E|_{V^i})/\pi_*(F^i_j E|_{V^i}) \cong  (\pi_*E)|_{U^i}/(\pi_* F^i_j E)|_{U^i}
 \end{equation}
 
The last isomorphism in \eqref{numericaltypeiso} is due to \cite[Proposition 9.3]{MR0463157} and we have $(\pi_*E)|_{U^i}/(\pi_* F^i_j E)|_{U^i}$ 
      is isomorphic with $ \pi_* E/\pi_*F^i_j E$. Therefore the QPS $r^i_j$ remains same under $\pi_*$. Since we will assign the same weights $\alpha^i_j$ for the filtration \eqref{pushforwardofparabolicGie} the numerical type will remain the same.
\end{proof}

\begin{definition}
	\label{stabilityofparGiebun}
	The parabolic Gieseker bundle $E_*$ is called stable if $\pi_* E$ with respect to the parabolic structure \eqref{pushforwardofparabolicGie} is $p_2$-stable \eqref{p_2stability}.
\end{definition}

\begin{remark}
	Similarly if we define semistability of $E_*$ to be the $p_2$-semistability of $\pi_* E$, then it will not be a GIT semistability notion \eqref{semistablemoduli}. One needs to impose additional condition along with $\pi_* E$ to be $p_2$-semistable. A. Schmitt in \cite[Definition 2.2.10]{MR2106123} has worked out a notion of semistability for Gieseker vector bundles. But in this article we will restrict our attention to stable parabolic Gieseker vector bundles. 
\end{remark}
Let $\label{objectsofmoduli}\mathfrak{{\mathfrak S}}$ be the family 
\begin{center}
	$\mathfrak{{\mathfrak S}}$=\{``Aut- equivalance" classes of pairs $(C, E_*)$ where $C$ is a marked Gieseker curve and $E_*$ is a stable parabolic Gieseker bundle on $C$ of fixed numerical type : $C \rightarrow C'$ is the contraction morphism to the marked stable curve $[C'] \in \overline{M}_{_{g, n}}$, where $[C']$ denotes the isomorphism class of the curve $C'$\}
\end{center}

Two pairs $(C, E_*)$ and $(D, F_*)$ are called  Aut- equivalant if there exist a marked isomorphism $\psi:C \cong D$ such that we have a parabolic isomorphism $ E_* \cong  \psi^*F_*$.\\

Our aim is to give a scheme structure on the set $\mathfrak{{\mathfrak S}}$.

	\subsection{Boundedness}
	In this subsection we will prove that the objects in the set $\mathfrak{{\mathfrak S}}$ form a bounded family. In particular we will have the following result
	\begin{lemma}
		\label{Giesekerbddness}
		There exist $l_0 \gg 0$ such that for all $l \ge l_0$ we have the cohomology vanishing $\text{H}^1(C, E(l))=0$ and the natural map $\text{H}^0(C, E(l)) \otimes \mathcal{O}_{C} \rightarrow E(l)$ is  surjective for all $(C, E_*) \in \mathfrak{{\mathfrak S}}$.\\
		
		 Furthermore the natural morphism $C \rightarrow \text{Gr}(\text{H}^0(E(l)), r)$  is a closed embedding. 
	\end{lemma}

\begin{proof}
	We will denote the partial normalization $\overline{C-\cup R^j}$ of $C'$ by $\tilde{C'}$ and $\nu:\tilde{C'} \rightarrow C'$ is the normalization morphism. Then the inclusion $i: \tilde{C'} \hookrightarrow C$ is a closed immersion.
	\begin{equation}
	\label{partialnormalization}
		\begin{tikzcd}
		\tilde{C'} \arrow[rr, "i", hook] \arrow[dr,swap, "\nu"]
		& &C \arrow[dl, "\pi"]\\%
		& C'
		\end{tikzcd}
	\end{equation}
	
	  The family of sheaves \{$\pi_* E$:$E \in {\mathfrak S}$ \} is contained inside the family
	  
	  \begin{center}
	  	\bigg\{$p_2$ semistable sheaves $\mathcal{E}_*$ on $\mathcal{X}_s$ with fixed numerical type $(r, d, r^i_j, \alpha^i_j)$ :$s \in {\tt S_{_{g,n}}}$ \bigg\}
	  \end{center} 
  
   Hence by \eqref{boundedness} there exist $l_0 \gg 0$ such that for all $l \ge l_0$ we will have the following
   
   \begin{center}
   	$(1)$ The cohomology vanishing $\text{H}^1(C', \pi_*E(l))=0$.
   \end{center}
	  
	 \begin{center}
	 $(2)$ The natural map $\text{H}^0(C', \pi_*E(l)) \otimes \mathcal{O}_{C'} \rightarrow \pi_*E(l)$ is  surjective i.e., $\pi_*E(l)$ is globally generated.
	 \end{center}
	for all $(C, E_*) \in \mathfrak{{\mathfrak S}}$\\
	
	Since $\text{H}^i(C, E(l)) \cong \text{H}^i(C', \pi_*E(l))$, we have $\text{H}^1(C, E(l))=0$ for $l \ge l_0$.\\
	
	Our aim is to show that $\{E(l)|_{\tilde{C'}}:E \in {\mathfrak S}\}$ is a bounded family. Then the lemma will follow from  \eqref{boundednessproposition}.\\
	
	Since the diagram \eqref{partialnormalization} commutes, we have $\nu=\pi \circ i$ and so $\nu^* \pi_* E(l) = i^* \pi^* \pi_* E(l)$. From the adjoint property we get a natural map $ \pi^* \pi_* E \rightarrow E$.  Applying $i^*$ will give us the following short exact sequence
	\begin{equation}
		0 \rightarrow \nu^* \pi_* E(l)/\text{Tor}(\nu^* \pi_* E(l)) \rightarrow E(l)|_{\tilde{C'}} \rightarrow \mathcal{Q} \rightarrow 0
	\end{equation}
	since $\nu $ is an isomorphism except for finitely many points $\mathcal{Q}$ is a torsion sheaf. We will denote $\nu^* \pi_* E(l)/\text{Tor}(\nu^* \pi_* E(l))$ by $\tilde{E}$. So it is enough to show $\tilde{E}$ is globally generated and $\text{H}^1(\tilde{C'}, \tilde{E})=0$.\\
	
	By definition of $\tilde{E}$ we have the quotient map $\nu^* \pi_* E(l) \rightarrow \tilde{E}$. Applying $\nu_*$ and using the natural map $\pi_* E(l) \rightarrow \nu_* \nu^* \pi_* E(l)$, we will get the following short exact sequence on $C'$ 
	\begin{equation}
		0 \rightarrow \pi_*E(l) \rightarrow \nu_* \tilde{E} \rightarrow \mathcal{Q'} \rightarrow 0
	\end{equation}
	where $\mathcal{Q'}$ is a torsion sheaf. Since $\pi_* E(l)$ is globally generated, by the five lemma $\nu_* \tilde{E}$ is globally generated. We also get $\text{H}^1(C', \nu_* \tilde{E})=0$ and $\text{H}^0(\tilde{C'}, \tilde{E}) \cong \text{H}^0(C', \nu_* \tilde E)$. The sheaf $\tilde{E}$ is locally free at $p_1^j, p_2^j$. We can check the equality $(\nu_* \tilde{E})_{z^j}=(\tilde{E})_{p_1^j} \oplus (\tilde{E})_{p_2^j}$. Hence the global generation of $\nu_* \tilde{E}$ will imply that $\tilde{E}$ is globally generated.
\end{proof}
	
	\subsection{Relative divisors on families of semi-stable curves}
\label{reldivGie}
We fix a natural number $l \ge l_0$ such that proposition \eqref{Giesekerbddness} holds. It follows from proposition \eqref{Giesekerbddness} that $\text{dim} \ \text{H}^0(C, E(l))$ is independent of the pair $(C, E) \in \mathcal{{\mathfrak S}}$ since $r$, $d$ are fixed. We will fix a vector space $V_l$ over $\mathbb{C}$ such that $\text{dim}_{\mathbb{C}}\text{H}^0(C, E(l)) = \text{dim}_{\mathbb{C}} V_l$. We will rigidify the moduli problem by adding a new data, namely an isomorphism $\text{H}^0(C, E(l)) \cong V_l$.
Let $\text{Gr}(V_l, r)$ be the Grassmannian of $r$ dimensional quotients of the $\mathbb{C}$ vector space $V_l$. Let $\mathcal{G}=\mathcal{G}(r, d)$ be the functor \eqref{Giefunctor}. Since the map $i:C \hookrightarrow \text{Gr}(V_l, r)$ is a closed embedding \eqref{Giesekerbddness}, it follows that the map $\pi \times i:C \hookrightarrow C' \times \text{Gr}(V_l, r)$ is also a closed embedding. The vector bundle $E(l)$ on $C$ is the pullback of the tautological quotient bundle on $\text{Gr}(V_l, r)$. This shows the functor \eqref{Giefunctor} is natural to consider in this context.  The family $\mathcal{X}/{\tt S_{_{g,n}}}$ has natural sections $\sigma_i:{\tt S_{_{g,n}}} \rightarrow \mathcal{X}$ which will give us divisors on the semi-stable curves $C$.\\
 
 Let $D^i$ be the associated divisors corresponding to sections $\sigma_i$, $1 \le i \le n$. So $D^i \rightarrow {\tt S_{_{g,n}}}$ is flat. Let $T$ be a ${\tt S_{_{g,n}}}$ scheme. By the base change $T \rightarrow {\tt S_{_{g,n}}}$ we get the flat morphism $D^i \times_{\tt{S}} T \rightarrow T$. We will denote $D^i \times_{\tt{S}} T$ by $D^i_T$. So $D^i_T$ are relative divisors\footnote{The word ``relative" in relative divisors means the map $D^i_T \rightarrow T$ is flat} of $\mathcal{X} \times_{\tt{S}} T$. \\
 
 Let $\Delta \in \mathcal{G}(T)$. We consider the induced morphism $\pi_T:\Delta \rightarrow \mathcal{X} \times _S T$ which is the collapsing morphism of the family of curves \{$\Delta_t: t \in T$\}. The components of each fiber, where the restriction of the relative dualizing sheaf $\omega_{\Delta/T}$ is trivial, are getting contracted. So the morphism $\pi_T$ is a birational morphism.
 \begin{equation}
 		\begin{tikzcd}[column sep=4em,row sep=4em]
 		D^i_T \times_{\left(\mathcal{X} \times_S T\right)} \Delta \rar [hook, "closed"] \dar\drar[phantom, "\square"] & \Delta \dar \\%
 		D^i_T \rar[swap, hook, "closed"] & \mathcal{X} \times_S T 
 		\end{tikzcd}
 	\end{equation}
 	
 Since for $t \in T$, the divisor $D^i_t$ is supported on the nonsingular locus of the curves $\mathcal{X}_t$, we see that $D^i_T$ sits inside the isomorphism locus of the birational map $\pi_T$. Thus the morphism $D^i_T \times_{\left(\mathcal{X} \times_S T\right)} \Delta \cong D^i_T$ is an isomorphism. \\
 
 We will use the same notation $D^i_T$ to denote the divisor $D^i_T \times_{\left(\mathcal{X} \times_S T\right)} \Delta$ on $\Delta$. The use of this notation should be clear from the context.

 \begin{remark}
 	\label{crucialremark}
 	The family $(\Delta, D^i_T)$ over $T$ is a family of marked Gieseker curves \eqref{markedGiecurve} (which are marked semi stable curves) with respect the morphism $\pi_{T}$. For any closed point $x \in \mathcal{X} \times_S T$, $\pi_T^{-1}(x)$ is either a singleton set or a connected chain of projective lines $R^j$. This map has the the following property:
 	\begin{equation}
 	(\pi_T)_* \mathcal{O}_{\Delta} \cong \mathcal{O}_{\mathcal{X} \times_S T}
 	\end{equation}
 	For a proof of this we refer to \cite[Proposition 6.7]{MR770932}.
 	\end{remark}

 \subsection{Parabolic Gieseker functor}
 \label{parGiefunctor}
 For any closed subscheme $\Delta \in \mathcal{G}(T)$ \eqref{Giefunctor}, let $E$ be be the pullback of the tautological quotient bundle on $\text{Gr}(V_l, r)$ to the closed subscheme $\Delta$. We define the following functor $\mathcal{G}(r, d, r^i_j)$
 \begin{equation}
 \label{Gieseker functor}
 	\mathcal{G}(r, d, r^i_j): \text{Sch}/{\tt S_{_{g,n}}} \rightarrow \text{Sets}
 	\end{equation}
where a ${\tt S_{_{g,n}}}$-scheme $T$ is sent to a  filtration of the bundle $E$ of the form
 	\begin{equation}
 	\label{relGiefiltration}
 		E=F^i_1 E\supset F^i_2 E \supset F^i_3 E \supset \cdots \supseteq F^i_{l_i+1} E=E(-D^i_T)
 	\end{equation} 
 such that the filtration has the following properties:\\
 
 $1.$ The quotients $E/F^i_j E$ are flat over $T$ $\;\forall \;i, \; j$. \\
 
 $2.$ For all $\;t \in T$ the restriction of the filtration \eqref{relGiefiltration} induces a filtration of the bundle $E_t$ on  $\Delta_t$ \eqref{qps} with respect to the divisors $D^i_t$.\\
 
 $3.$ The quotients $(E/F^i_j E)_t$ are supported on $D^i_t$ and are of dimension $r^i_j$.\\
 	
 	There is a forgetful morphism of functors $F:\mathcal{G}(r, d, r^i_j) \rightarrow \mathcal{G}$ that sends the parabolic vector bundles to the underlying vector bundles.\\

	\section{Some technical lemmas}
	\label{Some technical lemmas}
	In this section we will construct a flag variety $\mathfrak{F}_l$ which  represents the functor \eqref{parGiefunctor}. Then we will establish a relationship between $\mathfrak{F}_l$ and the flag variety $\text{F}_l$ for the torsion free parabolic moduli which represents the functor \eqref{flagfunctor}.\\
	
	We see that for all pair $(C, E_*) \in \mathfrak{S}$, the embedding of the curve $C$ in $\mathcal{X} \times \text{Gr}(V_l, r)$ has same Hilbert polynomial. We will denote the polynomial by $p(t)$. Let $\text{Hilb}^{p(t)}(\mathcal{X} \times \text{Gr}(V_l, r))$ be the relative Hilbert scheme over ${\tt S_{_{g,n}}}$. Let $T$ be a scheme over $\tt{S}_{_{g, n}}$ and $\Delta \in \mathcal{G}(T)$ \eqref{Giefunctor}. Then by definition \eqref{Giefunctor} we have the induced morphism $\pi_{T}:\Delta \rightarrow \mathcal{X} \times _S T$. Then the second condition of the definition is equivalent to $\Delta_t$ being a prestable curve of genus $g$ and $\omega_{\Delta_t} \cong \pi_{t}^* \omega_{\mathcal{X}_s}$, where $t$ maps to $s$ and $\pi_{t}:\Delta_t \rightarrow \mathcal{X}_s$ is the restriction of the morphism $\pi_{T}$. Both of these two conditions are open conditions i.e., \{t $\in$ T: $\Delta_t$ is a prestable curve of genus $g$ and $\omega_{\Delta_t} \cong \pi_{t}^* \omega_{\mathcal{X}_s}$\} is an open subset of $T$ \cite[p. 179]{MR739786}.\\
	
	Let $0 \rightarrow \mathcal{K} \rightarrow V_l \otimes \mathcal{O}_{\Delta} \rightarrow E \rightarrow 0$ be the pullback of the tautological short exact sequence on the Grassmannian $\text{Gr}(V_l, r)$. Then the third condition in \eqref{Giefunctor} is equivalent to the conditions $\text{dim}\text{H}^0(\mathcal{K}_t) \le 0$, $\text{dim}\text{H}^0(E_t) \le \text{dim}(V_l)$ and $\text{dim}\text{H}^1(E_t) \le 0$. By the upper semicontinuity of cohomology, these three cohomological conditions are open conditions. Therefore there exist an open subvariety $\text{Y}_{_{g, n}}$ of $\text{Hilb}^{p(t)}(\mathcal{X} \times \text{Gr}(V_l, r))$ over the base ${\tt S_{_{g,n}}}$ which represents the Gieseker functor \eqref{Giefunctor}.
	\subsection{Flag variety}
	Let the closed subscheme $\Delta \hookrightarrow \text{Y}_{_{g, n}} \times_S \left(\mathcal{X} \times \text{Gr}(V_l, r) \right)$ be the universal curve. Let $V_l \otimes \mathcal{O}_{\Delta} \rightarrow \text{Q}$ be the universal quotient on $\Delta$. By the construction \eqref{reldivGie} there exist relative divisors $D^i$ on the universal curve $\Delta$ over $\text{Y}_{_{g, n}}$
	\begin{equation}
	\begin{tikzcd}
	D^i \arrow[rr, hook] \arrow[dr,swap, "\text{flat}"]
	& &\Delta \arrow[dl]\\%
	& \text{Y}_{_{g, n}}
	\end{tikzcd}
	\end{equation}
	
Let  $\text{Q}|_{D^i}$ be the restriction of the  rank $r$ universal bundle $\text{Q}$ on $D^i$. We will consider the relative flag variety over $\text{Y}_{_{g, n}}$ of locally free quotients of the vector bundle $\text{Q}|_{D^i}$ on $D^i$ of rank in decreasing order $(r^i_{l_i+1}, r^i_{l_i}, \cdots , r^i_2)$. We will denote this relative flag variety by $\mathfrak{F}_l^i$. Let $\Delta^i$ be the fiber product $\Delta \times_{\text{Y}} \mathfrak{F}_l^i$. Then there exist a universal family which is a filtration of sheaves on $\Delta^i$ of the form
	\begin{equation}
	\label{universalfil}
		P_1^* \text{Q} = F^i_1\text{Q} \supset F^i_2\text{Q} \supset \cdots \supset F^i_{l_i+1}\text{Q}=P_1^*\text{Q}(-D^i)
	\end{equation}
where $P_1:\Delta^i \rightarrow \Delta $ be the projection map.\\

We have a natural closed embedding of  $\mathfrak{F}_l^i$ inside a product of Grassmannians
\begin{equation}
	\label{Gieflagembedding}
	\mathfrak{F}_l^i \hookrightarrow \text{Y}_{_{g, n}} \times \left(\text{Gr}(V_l, r^i_{l_i+1}) \times \cdots \times \text{Gr}(V_l, r^i_2) \right)
\end{equation}	
Let $\mathfrak{F}_l$ be the fiber product $\mathfrak{F}_l^1 \times_{\text{Y}} \mathfrak{F}_l^2 \times_{\text{Y}} \cdots \times_{\text{Y}} \mathfrak{F}_l^n$. There will be $n$ universal filtrations on $\Delta \times_{\text{Y}} \mathfrak{F}_l$ which are the pullbacks of the filtrations \eqref{universalfil} on $\Delta^i$ corresponding to the $n$ divisors $D^i$. The pullbacks are taken for the projections $\Delta \times_{\text{Y}} \mathfrak{F}_l \rightarrow \Delta^i$ and will remain filtrations since the relative flag varieties $\mathfrak{F}_l^i$ are flat over $\text{Y}_{_{g, n}}$. The flag variety $\mathfrak{F}_l$ along with these $n$ universal filtrations will represent the functor \eqref{parGiefunctor}. We have the following closed embedding of the relative flag variety $\mathfrak{F}_l$
\begin{equation}
\label{totalGieflagembedding}
\mathfrak{F}_l \hookrightarrow \text{Y}_{_{g, n}} \times \left(\text{Gr}(V_l, r^1_{l_1+1}) \times \cdots \times \text{Gr}(V_l, r^1_2) \right) \times \cdots \times \left(\text{Gr}(V_l, r^n_{l_n+1}) \times \cdots \times \text{Gr}(V_l, r^n_2) \right)
\end{equation}
	\subsection{Group action}
	The group $\text{SL}(N)$ and $\text{SL}(V_l)$ acts on $\mathcal{X}$ and $\text{Gr}(V_l, r)$ respectively. Thus the product $\text{SL}(N) \times \text{SL}(V_l)$ induces a natural action on $\mathcal{X} \times \text{Gr}(V_l, r)$ and therefore on $\text{Hilb}^{p(t)}(\mathcal{X} \times \text{Gr}(V_l, r))$. Note that the open subvariety $\tt{Y}_{_{g, n}}$ is invariant under the action of $\text{SL}(N) \times \text{SL}(V_l)$. Thus $\text{SL}(N) \times \text{SL}(V_l)$ will induce a natural action on the flag variety $\mathfrak{F}_l$. In particular, let $g=(g_1, g_2)$ be an element in $\text{SL}(N) \times \text{SL}(V_l)$ and $([C], E_*)$ be in $\mathfrak{F}_l$ where $[C] \in \tt{Y}_{_{g,n}}$. Let $(g \cdot [C])=[D]$ i.e., $D$ is the image of the curve $C$ under the automorphism $g:\mathcal{X} \times \text{Gr}(V_l, r) \rightarrow \mathcal{X} \times \text{Gr}(V_l, r)$. Then the action of $g$ on $\mathfrak{F}_l$ is $g \cdot ([C], E_*)=([D], g_* E_*)$. The closed embedding \eqref{totalGieflagembedding} is $\text{SL}(N) \times \text{SL}(V_l)$ equivariant. 
	\subsection{Relation between flag varieties corresponding to the parabolic pure sheaves and parabolic Gieseker bundles}
	The universal curve $\Delta \hookrightarrow \text{Y}_{_{g, n}} \times_S \left(\mathcal{X} \times \text{Gr}(V_l, r) \right)$ induces the proper birational morphism(canonical contraction) $\pi:\Delta \rightarrow \text{Y}_{_{g, n}} \times_S \mathcal{X}$ which has the property $\pi_* \mathcal{O}_{\Delta} \cong \mathcal{O}_{\mathcal{X} \times_S \text{Y}}$\eqref{crucialremark}. We have the universal quotient $V_l \otimes \mathcal{O}_{\Delta} \rightarrow \text{Q}$ flat over $\text{Y}_{_{g, n}}$. Thus applying $\pi_*$ we will get the morphism of sheaves on $\text{Y}_{_{g, n}} \times_S \mathcal{X}$
	\begin{equation}
	\label{morbetweenGietor}
		V_l \otimes \mathcal{O}_{\mathcal{X} \times_S \text{Y}} \rightarrow \pi_* \text{Q}
	\end{equation}
	The above morphism can be shown to be surjective using the isomorphism $(\pi_* \text{Q})_t \cong (\pi_t)_* \text{Q}_t$ for $t$ in $\text{Y}_{_{g, n}}$ due to \eqref{NSlemma}. Also we have $\pi_* \text{Q}$ is flat over $\text{Y}_{_{g, n}}$. Therefore by the definition of the Quot scheme functor the quotient \eqref{morbetweenGietor} will induce a morphism $\theta:\text{Y}_{_{g, n}} \rightarrow \text{Q}^r_g(\mu, V_l, H)$.\\
	
	We will now state and prove a result which is crucial to prove the properness of certain morphism \eqref{propernessofflag}. 
	\begin{proposition}
		\label{fiberprodprop}
	The following is a fiber product diagram
	\begin{equation}
	\label{fiberprod}
		\begin{tikzcd}[column sep=4em,row sep=3em]
		\mathfrak{F}_l \rar  \dar["\eta"] \drar[phantom, "\square"] & \text{Y}_{_{g, n}}\dar["\theta"] \\%
		\text{F}_l \rar[swap] & \text{Q}^r_g(\mu, V_l, H)
		\end{tikzcd}
	\end{equation}	
	The natural morphism $\eta:\mathfrak{F}_l \rightarrow \text{F}_l$ is $\text{SL}(N)  \times \text{SL}(V_l) $ equivariant.
	\end{proposition}
	\begin{proof}
		Let $\underline{Q}$ be the usual Quot scheme functor which is reprsented by $Q^r_g(\mu, V_l, H)$ \eqref{uniformrank} and recall that $\underline{\text{F}_l}$ be the functor \eqref{flagfunctor} defining the flag variety $\text{F}_l$. Thus to prove $\mathfrak{F}_l \cong \text{Y}_{_{g, n}} \times_{\text{Q}^r_g(\mu, V_l, H)} \text{F}_l$, it is enough to prove the following isomorphism between functors:
		\begin{equation}
		\label{isomoffunctors}
		\begin{tikzcd}
		\mathcal{G}(r, d, r^i_j) \arrow[r,shift left,"T_1"]
		&
		\mathcal{G} \times_{\underline{Q}} \underline{F_l} \arrow[l,shift left,"T_2"]
		\end{tikzcd}
		\end{equation}
		such that $T_1 \circ T_2= \text{id}$ and $T_2 \circ T_1= \text{id}$.\\
		
		Given a $S$ scheme $T$, let $(\Delta, E_*)$ be an element in  $\mathcal{G}(r, d, r^i_j)(T)$ with $\Delta \in \mathcal{G}(T)$. We define the morphism $T_1$ as
		\begin{equation}
		T_1((\Delta, E_*))=(\Delta, (\pi_{T})_* E_*)
		\end{equation} 
		We need to prove that $(\pi_{T})_* E_* \in \underline{F_l}(T)$ to make the definition well defined. For that we have to check $(\pi_{T})_* E_*$ satisfies the conditions in the definition \eqref{flatfamilyofparabolicsheaves}. \\
		
		$\RNum{1}.$  We will show that the quotients $(\pi_T)_* E/(\pi_T)_* F^i_j E$ are flat over $T$. By definition the quotients
		$E/F^i_j E$ are flat over $T$ and is supported on $D^i_T$. Since $\pi_T:\Delta \rightarrow \mathcal{X}_{T}$ restricted to $D^i_T$ is an isomorphism, $(\pi_T)_* \left( E/F^i_j E \right)$ is flat over $T$. Therefore it is enough to show that $(\pi_T)_* \left( E/F^i_j E \right) \cong (\pi_T)_* E/(\pi_T)_* F^i_j E$.\\
		
		 From the filtration of $E_*^*$ we get the short exact sequence
		\begin{equation}
		0 \rightarrow F^i_j E \rightarrow E \rightarrow E/F^i_j E \rightarrow 0
		\end{equation}
		Applying the pushforward $(\pi_T)_*$ we get
		\begin{equation}
		0 \rightarrow (\pi_T)_* F^i_j E \rightarrow (\pi_T)_*E \rightarrow (\pi_T)_*(E/F^i_j E) \rightarrow R^1 (\pi_T)_* F^i_j E
		\end{equation}
		
		We want to show $R^1 (\pi_T)_* F^i_j E=0$. The coherent sheaf $ R^1 (\pi_t)_* (F^i_j E)_t$ is the sheaf associated to the graded module $\oplus_{n \ge 0}\text{H}^1(\Delta_t, (F^i_j E)_t(n))$ where tensor product is taken with respect the line bundle $\mathcal{O}_{\Delta_t}(1)=\pi^*\mathcal{O}_{\mathcal{X}_s}(1)$ \eqref{S_{g, n} notation}. Since by definition $E_*$ induces a filtration $(E_t)^*_*$ when restricted to $\Delta_t$, we get the short exact sequence
		\begin{equation}
		\label{higherdirectzero}
		0 \rightarrow (F^i_j E)_t \rightarrow E_t \rightarrow  \text{Q} \rightarrow 0
		\end{equation}
		where the quotients of the filtration $(E_t)^*_*$ is supported at the divisors $x^i$ and we have $\text{Q} \subseteq E_t|_{x^i}$. Since $E_t$ is globally generated and $\text{H}^1(E_t)=0$, the cohomology long exact sequence corresponding to the short exact sequence \eqref{higherdirectzero} will give $\text{H}^1(\Delta_t, (F^i_j E)_t)=0$. Using similar arguments we will have $\text{H}^1(\Delta_t, (F^i_j E)_t(n))=0$ for all $n \ge 1$. Therefore the sheaf $ R^1 (\pi_t)_* (F^i_j E)_t=0$ for all $t \in T$.  By the lemma \eqref{NSlemma} we have the isomorphism $(R^1 (\pi_T)_* F^i_j E)_t \cong R^1 (\pi_t)_* (F^i_j E)_t=0$. Since each fiber over $t \in T$ vanishes, the sheaf $R^1 (\pi_T)_* F^i_j E$ is the $0$-sheaf.\\
		
		$\RNum{2}.$ Corresponding to the following base change diagram
		\begin{equation}
		\label{fiberprod}
		\begin{tikzcd}[column sep=4em,row sep=2em]
		\Delta_t \rar[hook]  \dar["\pi_t"] \drar[phantom, "\square"] & 
		\Delta \dar["\pi_T"] \\%
		\mathcal{X}_t \rar[swap, hook] & \mathcal{X}_T
		\end{tikzcd}
		\end{equation}	
		we will have the commutative diagram:
		\begin{equation}
		\begin{tikzcd}[column sep=2em,row sep=2em]
		(\pi_t)_*E_t  \arrow[d,"\cong"]&
		(\pi_t)_* (F^i_2 E)_t    \arrow[l, hook] \arrow[d,"\cong"] &
		\cdots  \arrow[l, hook] &
		(\pi_t)_*(F^i_{l_i}E)_t \arrow[l,hook]    \arrow[d,"\cong"] &
		(\pi_t)_*E_t(-D^i_t) \arrow[l,hook] \arrow[d,"\cong"]
		\\
		(\pi_*  E)_t  &
		(\pi_* F^i_2 E)_t      \arrow{l} &
		\cdots  \arrow{l} &     
		(\pi_* F^i_{l_i} E)_t \arrow{l} &
		(\pi_* E)_t(-D^i_t)  \arrow{l}
		\end{tikzcd}
		\end{equation}	
		By definition of $E_*$ the upper row is a filtration. By our argument in the above paragraph the coherent sheaf $ R^1 (\pi_t)_* (F^i_j E)_t=0$. Therefore by lemma \eqref{NSlemma} the vertical arrows are isomorphisms. By the commutativity of the diagram the lower row is also a filtration.\\
		
		Now we are going to \textbf{define the morphism $T_2$} \eqref{isomoffunctors}. Let $(\Delta, \mathcal{E}_*) \in \mathcal{G} \times_{\underline{Q}} \underline{F_l}(T)$. Let $E$ be the natural vector bundle on $\Delta$	such that $(\pi_T)_* E=\mathcal{E}$ which also implies $(\pi_T)_* E(-D^i_T)=\mathcal{E}(-D^i_T)$. The filtration of sheaves $\mathcal{E}_*$ is equivalent to the sequence of quotients:
		\begin{equation}
		\label{seqofquo}
		\mathcal{E}/\mathcal{E}(-D^i_T)=\mathcal{E}|_{D^i_T} \rightarrow \mathcal{E}/(F^i_{l_i} \mathcal{E}) \rightarrow \cdots \rightarrow \mathcal{E}/(F^i_2 \mathcal{E})
		\end{equation}
		Since $R^1(\pi_T)_* E(-D^i_T)=0$, from the short exact sequence
		\begin{equation}
		0 \rightarrow E(-D^i_T) \rightarrow E \rightarrow E|_{D^i_T} \rightarrow 0
		\end{equation}
		it follows $(\pi_T)_* E|_{D^i_T}=\mathcal{E}|_{D^i_T}$. Since $\pi_T$ induces canonical isomorphism on $D^i_T$, we will construct the following sequence of quotients on $\Delta$ from \eqref{seqofquo}
		\begin{equation}
		E|_{D^i_T}=(\pi_T^{-1})_* \mathcal{E}|_{D^i_T} \rightarrow (\pi_T^{-1})_*\left( \mathcal{E}/F^i_{l_i} \mathcal{E} \right) \rightarrow \cdots \rightarrow (\pi_T^{-1})_* \left(\mathcal{E}/F^i_2 \mathcal{E}\right)
		\end{equation}
		which is equivalent to a filtration of sheaves $E_*$. We define
		\begin{equation}
		T_2(\Delta, \mathcal{E}_*)=(E_*)
		\end{equation}
		
		$(\RNum{1})$ It is clear from construction of the filtration $E_*$ that quotients are flat sheaves over $T$ since they are isomorphic with $\mathcal{E}/F^i_j \mathcal{E}$. \\
		
	$(\RNum{2})$ To prove that the restriction of $E_*$ on $\Delta_t$ induces filtration, we can use similar arguments that has been used to prove property $(\RNum{2})$ for the morphism $T_1$.\\
		
		 From the definition of $T_1$ and $T_2$ it is clear $T_1 \circ T_2=\text{id}$ and $T_2 \circ T_1=\text{id}$.
	\end{proof}
	\section{moduli construction}
	\label{moduli construction}
	\subsection{Properness}
	The image of the morphism $\theta$ lands inside the open subscheme $\text{R}$ of $\text{Q}^r_g(\mu, V_l, H)$ consisting of points $q \in \text{Q}^r_g(\mu, V_l, H)$ where $q$ maps to $s \in {\tt S_{_{g,n}}}$ such that the quotient map $V_l \otimes \mathcal{O}_{\mathcal{X}_s} \rightarrow \mathcal{U}_q$ induces an isomorphism $V_l \cong \text{H}^0(\mathcal{X}_s, \mathcal{U}_q)$. We denote by $\text{R}^f$ the open subscheme of $\text{R}$ such that the sheaf $\mathcal{U}_q$ is torsion free on $\mathcal{X}_s$. Let $\text{Y}^f:=\theta^{-1}(\text{R}^f)$.\\
	
	Let $C$ be a smooth curve and $\zeta:C \rightarrow \tt{S}_{_{g, n}}$ be a morphism. We denote the base change $\mathcal{X} \times_{\tt{S}} C$ by $\mathcal{X}_{_{C}}$. Further we assume that $\mu^{-1}(p)=\mathcal{X}_p$ for $p \in C$ is the only singular fiber of the family $\mu:\mathcal{X}_{_{C}} \rightarrow C$.
	\begin{lemma}{\cite[Proposition $4.2$]{MR3579973}}
	\label{cruciallemmaforproperness}
	Let $\mathcal{E}_{_{C}}$ be a flat family of torsion free sheaves on $\mathcal{X}_{_{C}}$. Then there exist a family of marked Gieseker curves  $\mathcal{X}'_{_{C}}$ with the canonical contraction $\pi:\mathcal{X}'_{_{C}} \rightarrow \mathcal{X}_{_{C}}$ such that $E_{_{C}}=\left(\pi^* \mathcal{E}_{_{C}}/\text{Tor}\right)$ obtained by going modulo torsion is a family of Gieseker vector bundles and furthermore we will also have an isomorphism $\pi_*(E_{_{C}}) \cong \mathcal{E}_{_{C}}$. 
	\end{lemma}
\begin{proof}
	 All of the isolated singular points of the surface $\mathcal{X}_{_{C}}$ will lie inside the only singular fiber $\mathcal{X}_p$. For a singular point $u \in \mathcal{X}_{_{C}}$ we will take formal a neighbourhood which will be of the form $\frac{\mathbb{C}[[x, y, t]]}{(xy-t^m)}$. Then restriction of the family of torsion free sheaves $\mathcal{E}_{_{C}}$ will be 
	\begin{equation}
		\mathcal{E}_{_{C}}|_{U} \cong \mathcal{O}_{U}^j \oplus \bigoplus_{1}^{L}(x, t^{\beta_i})^{\oplus m_i}
	\end{equation}
	where $j+\sum m_i=\tt{rank}(\mathcal{E}_{_{C}})$. Then to construct the surface $\mathcal{X}'_{_{C}}$ we have to blow up the surface $U$ along the ideal sheaves $(x, t^{\beta_i})^{\oplus m_i}$. 
	  For details of these blow ups we refer to \cite[Proposition $4.2$]{MR3579973}. 
\end{proof}

	\begin{proposition}
		\label{properness}
		The morphism $\theta:\text{Y}^f \rightarrow \text{R}^f$ is proper.
	\end{proposition}

\begin{proof}
	Similar properness results has been proved in \cite[Proposition $10$]{MR1687729}, \cite[Theorem $2.1.2$]{MR2106123}. Since the map $\theta$ is quasi projective in a relative set up (over $\tt{S}_{_{g, n}}$), this allows us to use a particular form of valuative criterion, called horizontal properness \cite[Definition 4.4]{MR3579973}.\\
	
Since the morphism $\theta$ is quasi projective, there exist a projective morphism $\bar{\theta}:\text{Z} \rightarrow \text{R}^f$ with the following diagram
	  \begin{equation}
	  \label{propernessdiagram}
	  \begin{tikzcd}
	  \text{Y}^f  \arrow[rr, hook, "open"] \arrow[dr, "\theta"]
	  & &\text{Z}\arrow[dl, "projective", "\bar{\theta}"']\\%
	  & \text{R}^f
	  \end{tikzcd}
	  \end{equation}
	  
	  Let ${\tt S_{_{g,n}}^0}$ be the open subvariety of ${\tt S_{_{g,n}}}$ representing the locus of marked nonsingular curves. Let $\pi:D \rightarrow D'$ be a morphism of prestable curves with equal genus such that $\pi^* \omega_{C'} \cong \omega_{C}$ and $D'$ be a nonsingular curve. This will imply that $\pi$ is an isomorphism. Therefore by the definition of the functor $\mathcal{G}$ \eqref{Gieseker functor}, $\theta|_{{\tt S_{_{g,n}}^0}}: \text{Y}^f|_{{\tt S_{_{g,n}}^0}} \cong \text{R}^f|_{{\tt S_{_{g,n}}^0}}$ is an isomorphism. Thus in particular it is a proper morphism. We have $ \text{R}^f|_{{\tt S_{_{g,n}}^0}}$ is an open subvariety of $\text{R}$. we will take the closure $\overline{\text{Y}^f|_{{\tt S_{_{g,n}}^0}}}$ inside $\text{Z}$. Without loss of generality we can assume $\text{Z}= \overline{\text{Y}^f|_{{\tt S_{_{g,n}}^0}}}$. Thus to prove properness of $\theta$ it is enough to show $\text{Y}^f=\overline{\text{Y}^f|_{{\tt S_{_{g,n}}^0}}}$ \eqref{propernessdiagram}.\\
	  
	  Let $x$ be a point in $\overline{\text{Y}^f|_{{\tt S_{_{g,n}}^0}}} \setminus \text{Y}^f|_{{\tt S_{_{g,n}}^0}}$, then we can assume there exist a smooth curve $C$ with a morphism $\kappa:C \rightarrow \overline{\text{Y}^f|_{{\tt S_{_{g,n}}^0}}}$  such that   $\kappa(C\setminus p) \subseteq \text{Y}^f|_{{\tt S_{_{g,n}}^0}}$ and $\kappa(p)=x$ for some point $p \in C$.
	This will induce the map $\bar{\theta} \circ \kappa:C \rightarrow \text{R}^f$. We will show that there exist a morphism $\kappa':C \rightarrow \text{Y}^f$ with the following commutative diagram
	\begin{equation}
	\begin{tikzcd}[column sep=4em,row sep=2em]
	C \setminus p \ar[r, "\kappa"] \ar[d, swap, hook] & \text{Y}^f \ar[d, "\theta"] \\
	C
	\ar[ur, shift right = .75ex, swap, dashed, "\kappa'"]
	\ar[r, "\bar{\theta} \circ \kappa"]
	&
	\text{R}^f 
	\end{tikzcd}
	\end{equation}
	
	The map $\bar{\theta} \circ \kappa$ will be induced by the flat family of torsion free sheaves $\mathcal{E}_{_{C}}$ on family of marked stable curves $\mathcal{X}_{_{C}}(=\mathcal{X} \times_S C)$ over $C$ along with a quotient representation $V_l \otimes \mathcal{O}_{\mathcal{X}_{_{C}}} \rightarrow \mathcal{E}_{_{C}}$ . Hence by \eqref{cruciallemmaforproperness} we will get the following quotient which is flat over $C$
	\begin{equation}
	\label{givesmaptoGrass}
	V_l \otimes \mathcal{O}_{\mathcal{X'}_{_{C}}} \rightarrow {{E}_{_{C}}}
	\end{equation}
	where $\mathcal{X'}_{_{C}}$ is a family of marked Gieseker curves and ${E}_{_{C}}$ is a family of Gieseker bundles over $C$.\\
	
	The locally free quotient in \eqref{givesmaptoGrass} will define a morphism to the Grassmannian
	\begin{equation}
	\label{maptoGrass}
		\mathcal{X'}_{_{C}} \rightarrow C \times \text{Gr}(V_l, r) 
	\end{equation}
	Let $\pi:\mathcal{X'}_{_{C}} \rightarrow \mathcal{X}_{_{C}}$ is the contraction morphism. We also have $\pi_* E_{_{C}} \cong \mathcal{E}_{_{C}}$ \eqref{cruciallemmaforproperness}. Hence the family of bundles $E_{_{C}}$ on $\mathcal{X'}_{_{C}}$ satisfies the conditions of \eqref{Giesekerbddness}. Therefore the morphism \eqref{maptoGrass} is a closed embedding. It is clear that the family $\mathcal{X'}_{_{C}}$  satisfies the conditions of the Gieseker functor \eqref{Gieseker functor} . Thus the family will define a morphism $\kappa':C \rightarrow \text{Y}^f_{_{g, n}}$ which will agree with the morphism $\kappa$ on $C\setminus p$. Since $ \overline{\text{Y}^f|_{{\tt S_{_{g,n}}^0}}}$ is a separated scheme we will have $\kappa'=\kappa$. Therefore we have $\text{Y}^f_{_{g, n}}=\overline{\text{Y}^f|_{{\tt S_{_{g,n}}^0}}}$.
\end{proof}

	Let $\text{F}_l^f$ is the base change $\text{F}_l \times_ {\text{Q}^r_g(\mu, V_l, H)} \text{R}^f$ and $\mathfrak{F}_l^f$ is the base change $\mathfrak{F}_l \times_{\text{Y}_{_{g, n}}} \text{Y}^f_{_{g, n}}$. By Proposition \eqref{fiberprodprop} we have $\mathfrak{F}_l \cong \text{F}_l \times_ {\text{Q}^r_g(\mu, V_l, H)} \text{Y}_{_{g, n}}$. Therefore we have $\mathfrak{F}_l^f=\mathfrak{F}_l \times_{\text{Y}_{_{g, n}}} \text{Y}^f_{_{g, n}} \cong \left(\text{F}_l \times_ {\text{Q}^r_g(\mu, V_l, H)} \text{Y}_{_{g, n}}\right)\times_{\text{Y}_{_{g, n}}} \text{Y}^f_{_{g, n}} \cong  \text{F}_l \times_ {\text{Q}^r_g(\mu, V_l, H)} \text{Y}^f_{_{g, n}} \cong \text{F}_l^f \times_{\text{R}^f} \text{Y}^f_{_{g, n}}$ i.e., we have a modified cartesian product diagram
	\begin{equation}
	\begin{tikzcd}[column sep=4em,row sep=1em]
	\mathfrak{F}_l^f \rar  \dar["\eta"] \drar[phantom, "\square"] & \text{Y}^f_{_{g, n}}\dar["\theta"] \\%
	\text{F}_l^f \rar[swap] & \text{R}^f
	\end{tikzcd}
	\end{equation}
	
	\begin{corollary}
		\label{propernessofflag}	
	 The morphism $\eta:\mathfrak{F}_l^f \rightarrow \text{F}_l^f$ is proper {\em (by lemma \eqref{properness})}. Let $\text{F}_l^s$ be the open subscheme of $p_2$-stable sheaves in $\text{F}_l^f$. Let $\mathfrak{F}_l^s=\eta^{-1}(\text{F}_l^s)$ be the parabolic vector bundles on marked semistable curves whose pushforward is $p_2$-stable. {\em Then (again by base change we see) the morphism $\eta:\mathfrak{F}_l^s \rightarrow \text{F}_l^s$ is proper.}
\end{corollary}
\vspace{0.3cm}

	\subsection{Proof of main theorem \eqref{mainthm1}}
	
	The proof essentially follows the methods used in \cite[p. 180]{MR1687729}. But here we have to deal with few more complexities. For the sake of completeness  we will give the argument. \\
	
	 The first step is to embed $\mathfrak{F}_l^f$ as a locally closed subscheme in a projective variety. We have the following locally closed embedding \eqref{totalGieflagembedding}
	\footnotesize \begin{equation}
	\label{F_l^{gie}embed2}
	\mathfrak{F}_l^f \hookrightarrow \mathfrak{F}_l \hookrightarrow \text{Y}_{_{g, n}} \times \left(\text{Gr}(V_l, r^1_{l_1+1}) \times \cdots \times \text{Gr}(V_l, r^1_2) \right) \times \cdots \times \left(\text{Gr}(V_l, r^n_{l_n+1}) \times \cdots \times \text{Gr}(V_l, r^n_2) \right)
	\end{equation}\normalsize
	where the first embedding is an open embedding and the second one is a closed embedding. We will denote the product $\prod_{i=1}^{n} \prod_{j=2}^{l_i+1} \text{Gr}(V_l, r^i_j)$ by $\textbf{\text{Gr}(\text{F})}$.
	Since the family $\mathcal{X}$ has the closed embedding $\mathcal{X} \hookrightarrow {\tt S_{_{g,n}}} \times \P^{N}$ \eqref{localunifamily} we get the closed embedding
	\begin{equation}
		\label{F_l^{gie}embed4}
		 \text{Hilb}^{p(t)}(\mathcal{X} \times \text{Gr}(V_l, r) ) \hookrightarrow \text{Hilb}^{p(t)}({\text S_{_{g,n}}} \times \P^{N} \times \text{Gr}(V_l, r))
		 \end{equation}
		
	By the universal property of Hilbert scheme the relative Hilbert scheme $\text{Hilb}^{p(t)}({\tt S_{_{g,n}}} \times \P^{N} \times \text{Gr}(V_l, r))$ over ${\tt S_{_{g,n}}}$ is isomorphic with  ${\tt S_{_{g,n}}} \times_{\mathbb{C}} \text{Hilb}^{p(t)}(\P^{N}\times \text{Gr}(V_l, r))$. The inclusion of ${\tt S_{_{g,n}}}$ in its closure $\overline{{\tt S_{_{g,n}}}}$ is an open immersion\footnote{Since ${\tt S_{_{g,n}}}$ is a locally closed subscheme of $I$ \eqref{incidence},  ${\tt S_{_{g,n}}}$ is open inside $\overline{{\tt S_{_{g,n}}}}$ with the closure being taken in $I$} and $\text{Y}$ is an open subscheme of $\text{Hilb}^{p(t)}(\mathcal{X} \times \text{Gr}(V_l, r))$. Hence as a composition of finitely many locally closed immersion we have the locally closed immersion of $\mathfrak{F}_l^f$ in a projective variety
	\begin{equation}
	\label{F_l^{gie}embed5}
		\mathfrak{F}_l^f \hookrightarrow \overline{{\tt S_{_{g,n}}}} \times_{\mathbb{C}} \text{Hilb}^{p(t)}(\P^{N}\times \text{Gr}(V_l, r)) \times_{\mathbb{C}} \textbf{\text{Gr}(\text{F})} 
	\end{equation}
	
The embedding \eqref{F_l^{gie}embed2} is $\text{SL}(N) \times \text{SL}(V_l)$ equivariant and the closed embedding $\mathcal{X} \hookrightarrow {\tt S_{_{g,n}}} \times \P^{N}$ is $\tt{SL}(N)$ equivariant. Therefore the embedding of $\mathfrak{F}_l^f$ \eqref{F_l^{gie}embed5} is $\text{SL}(N) \times \text{SL}(V_l)$ equivariant embedding. \\

Since \eqref{F_l^{gie}embed5} is a locally closed embedding, $\mathfrak{F}_l^f$ is open in $\overline{\mathfrak{F}_l^f}$ with the closure being taken inside the projective variety at the R.H.S of \eqref{F_l^{gie}embed5}. We denote the closure $\overline{\mathfrak{F}_l^f}$ by $Z_1$.\\

 Let $\mathcal{O}_{Z_1}(1)$ be an ample linearization of the $\text{SL}(N) \times \text{SL}(V_l)$ action on $Z_1$. Recall the embedding \eqref{totalspaceF_lembedding} of the flag variety $\text{F}_l$. We get the following equivariant commutative diagram:
\begin{equation}
\begin{tikzcd}[column sep=4em,row sep=2em]
\mathfrak{F}_l^f \arrow[r, hook]  \dar["\eta"]  & Z_1 \dar[dashed] \\%
\text{F}_l^f \rar[swap, hook] & \overline{{\tt S_{_{g,n}}}} \times \textbf{\text{Gr}(l, k)}
\end{tikzcd}
\end{equation} 
Let $Z$ be the graph closure of the vertical rational morphism. Then $Z \hookrightarrow Z_1 \times \overline{{\tt S_{_{g,n}}}} \times\textbf{\text{Gr}(l, k)}$ is a projective variety. Recall the linearization $\overline{L}=L_{\beta, \beta^i_j} \boxtimes \mathcal{M}^{\otimes b}$  \eqref{linearizationproductaction} of the $\text{SL}(N) \times \text{SL}(V_l)$ action on $\overline{{\tt S_{_{g,n}}}} \times \textbf{\text{Gr}(l, k)}$. We will choose the linearization $\overline{L}^{\otimes a} \boxtimes \mathcal{O}_{Z_1}(1)$ on $Z$. We have the following equivariant commutative diagram
\begin{equation}
\label{equivariantdiagram}
\begin{tikzcd}[column sep=4em,row sep=2em]
\mathfrak{F}_l^f \arrow[r, hook]  \dar["\eta"]  & Z \dar["\lambda"] \\%
\text{F}_l^f \rar[swap, hook] & \overline{{\tt S_{_{g,n}}}} \times \textbf{\text{Gr}(l, k)}
\end{tikzcd}
\end{equation}
 where $\lambda$ is the $2^{\text{nd}}$ projection.\\
 
The open locus of marked nonsingular curves ${\tt S_{_{g,n}}^0}$ is irreducible. We consider the restriction $\text{Y}^f|_{\tt S_{_{g,n}}^0}$. The family $\text{Y}^f|_{\tt S_{_{g,n}}^0} \rightarrow {\tt S_{_{g,n}}^0}$ is flat and the fibers are irreducible. Therefore the total space $\text{Y}^f|_{\tt S_{_{g,n}}^0}$ is irreducible. From the proof of properness \eqref{properness}  we have $\overline{\text{Y}^f|_{\tt S_{_{g,n}}^0}}=\text{Y}^f$. This implies $\text{Y}^f$ is irreducible. The fibers of the flat morphism $\mathfrak{F}_l^f \rightarrow \text{Y}^f$ are flag varieties. Thus $\mathfrak{F}_l^f$ is irreducible. So $Z$ is an irreducible projective variety. \\

Recall that combining \eqref{GITforSL(V_l)} and \eqref{tfmoduliinterpretation} will imply that a point $(C', \mathcal{E}_*, \gamma, V_l \otimes \mathcal{O}_{C'} \rightarrow \mathcal{E})$ in $\text{F}_l$ is $p_2$-stable and $\gamma:V_l \cong \text{H}^0(\mathcal{E}(l))$ is an isomorphism (i.e., the point belongs in the open locus of $p_2$-stable sheaves $\text{F}_l^s \subseteq \text{F}_l^f$) if and only if it's image in  $\overline{{\tt S_{_{g,n}}}} \times \textbf{\text{Gr}(l, k)}$ is GIT stable with respect to the linearization $\overline{L}$ of the $\text{SL}(N) \times \text{SL}(V_l)$ action.\\

With respect to the linearization $\overline{L}^{\otimes a} \boxtimes \mathcal{O}_{Z_1}(1)$ on $Z$ we apply the GIT lemma \eqref{GITlemma} for $a \gg0$. we get  $\lambda^{-1}(\text{F}_l^{s}) \hookrightarrow Z^{s}$ and the following diagram 
\begin{equation}
	\begin{tikzcd}
	\mathfrak{F}_l^{s} \arrow[rr, hook, "i"] \arrow[dr, "\eta"]
	& &\lambda^{-1}(\text{F}_l^{s})\arrow[dl, "\lambda"]\\%
	& \text{F}_l^{s}
	\end{tikzcd}
\end{equation}
Since the morphism $\lambda$ is the base change of the projective morphism $\lambda$ in \eqref{equivariantdiagram} it is proper. The morphism $\eta$ is proper \eqref{propernessofflag}. This implies $i$ is proper. Also $i$ is an open immersion. Irreducibilty of $\lambda^{-1}(\text{F}_l^s) \hookrightarrow Z^s$ means $i$ is an isomorphism. This means we have a GIT quotient $\mathfrak{F}_l^s \sslash (\text{SL}(N) \times \text{SL}(V_l))$ as a quasi projective variety, denoted by $\overline{\mathfrak{U}}_{_{g, n, r}}$. We also have a proper birational morphism $\overline{\mathfrak{U}}_{_{g, n, r}} \rightarrow \text{F}_l^{s} \sslash (\text{SL}(N) \times \text{SL}(V_l))=\overline{\mathcal{U}}_{_{g, n, r}}$.

\begin{remark}
	\label{semistablemoduli}
Note that by the GIT lemma \eqref{GITlemma} we have
\begin{equation}
	\lambda^{-1}(\text{F}_l^{s}) \subseteq Z^{s} \subseteq Z^{ss} \subseteq \lambda^{-1}(\text{F}_l^{ss})
\end{equation} 	
We will have a morphism $Z^{ss} \rightarrow \text{F}_l^{ss}$. The GIT quotient $Z^{ss} \sslash (\text{SL}(N) \times \text{SL}(V_l))$ which is a projective variety, can be called universal moduli space of semistable parabolic Gieseker vector bundles. In \cite[Definition 2.2.10]{MR2106123} a moduli theoretic interpretation of semistable locus is given in the context of vector bundles.\\

One way to give modular interpretation of points in $Z^{ss}$ is to make $\text{F}_l^{s}= \text{F}_l^{ss}$ which can be done by choosing weights $\alpha_j^i$ and quasi parabolic structures $r^i_j$ in such a way that $p_2$-stability=$p_2$-semistability. This can be done by choosing $r, d, r^i_j$ such that g.c.d $(r, d, r^i_j)=1$ and choosing suitable weights $\alpha^i_j$. In fact for generic weights $\alpha^i_j$ this will be true.
\end{remark}

	\section{properties of the moduli space $\overline{\mathfrak{U}}_{_{g, n, r}}$}
	\label{properties of the moduli space}
	In this section we will prove \eqref{mainthm2}, \eqref{mainthm3}. For the proofs we will need to study the deformation of a marked semistable curve.\\
	
	Let $C$ be a marked semistable curve and $\pi:C \rightarrow C'$ be the canonical contraction to it's stable model. There exist a formal universal deformation of $(C', x^1, \cdots, x^n)$
	\begin{equation}
	\label{formalunideformation}
	\begin{tikzcd}
	\mathfrak{C}  \arrow[d, "\mu"]   \\%
	\mathcal{M}=\mathbb{C}[[t_1, t_2, \cdots, t_{\tt{M}}]] \ar[u, bend left=50, "\sigma_i"]
	\end{tikzcd}
	\end{equation}
	where $\sigma_i$ for $1 \le i \le n$ are sections of the morphism $\mu$ and $\tt{M}=dim \left(Ext^1(\Omega^1_{C'}, \mathcal{O}_{C'}(-\tt{D}))\right)$. The divisor $\tt{D}$$=x^1 +\cdots +x^n $ is the divisor associated to the marked points \cite[pp. 79-80]{MR262240}. Let the nodes of the curve $C'$ be $\tt{p}_1, \tt{p}_2, \cdots ,\tt{p}_K$. Using Schlessinger's theory it can be proved $\hat{\mathcal{O}_{\mathfrak{C}, \tt{p}_i}} \cong \frac{\mathbb{C}[[u_i, v_i, t_1, \cdots, t_{M}]]}{(u_i \cdot v_i-t_i)}$ where $u_i, v_i$ are the local coordinates at the node $\tt{p}_i$ of the curve $C'$ \cite[p. 82]{MR262240}. \\
	 
	 Recall that $R^1, R^2, \cdots, R^m$ are the chains of $\P^1$'s in $C$ which are contracted by $\pi$ to the nodes $\tt{p}_1, \tt{p}_2, \cdots, \tt{p}_m$ respectively where $R^i= \cup_{j=1}^{\iota_i} R^i_j$. Let $\mathcal{N}=\ \text{Spec} \frac{\mathbb{C}[[t_1, t_2, \cdots, t_{\tt{M}}]]}{(t_{m+1}, \cdots, t_{\tt{M}})}$. Let $\mathfrak{C}_{\mathcal{N}}$ be the restriction of the family $\mathfrak{C}$ to the closed subscheme $\mathcal{N}$ of $\mathcal{M}$. By the local universal property \eqref{localuniproperty} there exist a morphism $\mathcal{N} \rightarrow \tt{S}_{_{g, n}}$ such that $\mathfrak{C}_{\mathcal{N}} \equiv \mathcal{X} \times_{\tt{S}_{_{g, n}}} \mathcal{N}$.
	 
	 Let $\mathcal{W}=\text{Spec}(\mathbb{C}[[t_{ij}:1 \le i \le m, \forall i, 1 \le j \le {\iota}_i]])$ and the morphism $\mathcal{W} \rightarrow \mathcal{N}$ be defined by $t_i \rightarrow t_{i1} \cdot t_{i2} \cdots \cdot t_{i \iota_i}$. 
	 \begin{lemma}\cite[Lemma 4.2]{MR739786} \cite[3.3.1]{MR2106123}
	 	There exist deformation $\mathcal{Z}_{\mathcal{W}}$ over $\mathcal{W}$ of the curve $C$ with a morphism $\psi:\mathcal{Z}_{\mathcal{W}} \rightarrow \mathcal{X} \times_{\tt{S}_{_{g, n}}} \mathcal{W}$ which restrict to the contraction morphism $\pi:C \rightarrow C'$ over the unique closed point of $\mathcal{W}$. The morphism $\psi$ has the property $\psi^* \omega_{_{\mathcal{X} \times_{\tt{S}_{_{g, n}}} \mathcal{W}/\mathcal{W}}} \cong \omega_{_{\mathcal{Z}_{\mathcal{W}}/\mathcal{W}}}$ relative to $\mathcal{W}$. The closed subscheme of $\mathcal{W}$, where the fibers of $\psi$ are singular, is defined by $\left(\prod_{j=1}^{\iota_i} t_{ij}=0, i=1, 2, \cdots, m \right)$.
	 \end{lemma}
	 
	  We will denote this subvariety of $\mathcal{W}$ by $\mathcal{Q}$. The family $\mathcal{Z}_{\mathcal{W}}$ over $\mathcal{W}$ \label{versalpropertyofdeformation} has the following versal property :\\
	  
	  \begin{proposition} \cite[Proposition 4.5]{MR739786}
	  	\label{versalproperty}
	  	Let $A$ be an Artin local ring over $\mathbb{C}$ such that we have a morphism $T=\tt{Spec}(A) \rightarrow \tt{S}_{_{g, n}}$. Let $\mathcal{Z}'$ over $T$ be a deformation of the marked semistable curve $C$ with a morphism $\psi':\mathcal{Z}' \rightarrow \mathcal{X} \times_{\tt{S}_{_{g, n}}} T$ over $T$ which restricts to the canonical contraction $\pi:C \rightarrow C'$ over the unique closed point of $T$. Then there exist a morphism $T \rightarrow \mathcal{W}$ such that the deformation $\mathcal{Z}'$ and the morphism $\psi'$ is isomorphic with $\mathcal{Z}''$ and $\psi''$ where $\mathcal{Z}''$ and $\psi''$ are the base change of $\mathcal{Z}_{\mathcal{W}}$ and $\psi$.
	  \end{proposition}

	\begin{proposition} \cite[Appendix I, III]{MR1687729}
		\label{Y is regular}
		The quasi projective variety $Y$ is regular. 
	\end{proposition}
\begin{proof}
	The variety $Y$ represents the functor \eqref{Giefunctor} and the universal family of curves is $\Delta \hookrightarrow Y \times_{\tt{S}_{_{g, n}}}(\mathcal{X} \times \text{Gr}(V_l, r))$. Let $y \in Y$ be a point which is represented by the curve $C$ i.e., $\Delta_y=C$. We can define a local Gieseker functor $\mathcal{G}_{\text{loc}}:\text{Artin local ring}/\tt{S}_{_{g, n}} \rightarrow \text{Sets}$ which sends a $\tt{S}_{_{g, n}}$ scheme $\text{Spec}(A)$ to $\Delta$ in $\mathcal{G}(\text{Spec}(A))$ such that the fiber of $\Delta$ over the closed point in $\text{Spec}(A)$ is the curve $C$. Let $\mathcal{G}_{C}:\text{Artin local ring}/\tt{S}_{_{g, n}} \rightarrow \text{Sets}$ be the functor which sends $\text{Spec}(A)$ to flat family of marked semistable curves $\Delta$ over $\text{Spec}(A)$ with a morphism $\psi:\Delta \rightarrow \mathcal{X} \times_{{S}_{_{g, n}}} \text{Spec}(A)$ such that the closed fiber is $C$ and $\psi^* \omega_{\mathcal{X}_{A}/A} \cong \omega_{\Delta/A}$  We have a natural forget morphism $F: \mathcal{G}_{\text{loc}} \rightarrow \mathcal{G}_{C}$ which can be proved to be formally smooth \cite[Appendix \RNum{1}]{MR1687729}. On the other hand using the versal property of the  deformation $\mathcal{Z}_{\mathcal{W}}$ over $\mathcal{W}$ it can be checked that ($\mathcal{Z}_{\mathcal{W}}, \mathcal{W}$) is a versal family for the functor $\mathcal{G}_{C}$. Hence $Y$ is regular at $y$.
\end{proof}
	
	\begin{proof}[Proof of Theorem \eqref{mainthm2}]
	The morphism $\mathfrak{F}_l \rightarrow Y$ is flat, projective and the fibers of the morphism are flag varieties of equal dimension. Therefore the morphism is smooth. Since the quasi projective variety $Y$ is regular \eqref{Y is regular}, we have $\mathfrak{F}_l$ is regular. Being an open subset of $\mathfrak{F}_l$, the stable locus $\mathfrak{F}_l^s$ is also regular. Thus $\mathfrak{F}_l^s \sslash \text{SL}(N) \times \text{SL}(V_l)$ is a normal quasi projective variety. We  have $Z^{ss}$ as an open subvariety of $\mathfrak{F}_l^{ss}$ \eqref{semistablemoduli}. Hence $Z^{ss}$ is regular. Therefore the semistable moduli $Z^{ss} \sslash \text{SL}(N) \times \text{SL}(V_l)$ is a normal projective variety.\\
	
	The map $\eta:\mathfrak{F}_l^s \rightarrow \text{F}_l^s$ is $\text{SL}(N) \times \text{SL}(V_l)$ equivariant. The stabilizer of a point in $\text{F}_l^s$ for the $\text{SL}(V_l)$ action is the finite cyclic group of order $\text{dim}(V_l)$. Therefore the stabilizer of a point in $\mathfrak{F}_l^s$ for the $\text{SL}(V_l)$ action is also the same finite group. Hence the $\text{SL}(V_l)$ action will factor through the $\text{PGL}(V_l)$ action which will act freely. The stabilizer of a point $(C, E_*)$ for $\text{SL}(N)$ action is the finite automorphism group $\text{Aut}(C', x^1, x^2, \cdots, x^n)$, where $C'$ is the marked stable model of the marked Gieseker curve $C$. Since the total space $\mathfrak{F}_l^s$ is regular, it follows from the \'etale slice theorem that the universal moduli space  $\mathfrak{F}_l^s \sslash \text{SL}(N) \times \text{SL}(V_l)$ has finite quotient singularity. \\
	
	A stable parabolic Gieseker bundle $E_*$ on a marked semistable curve $C$ is called strictly stable if for any automorphism $\phi:(C, x^1,\cdots, x^n) \rightarrow (C, x^1,\cdots, x^n)$ such that there exist a parabolic isomorphism $\phi^* E_* \cong E_*$ then $\phi=\tt{id}$. The strictly stable locus $\mathfrak{F}_l^*$ is an open subvariety of $\mathfrak{F}_l^s$. By definition the action of $ \text{SL}(N) \times \text{SL}(V_l)$ on $\mathfrak{F}_l^*$ is free. Therefore $\mathfrak{F}_l^* \rightarrow \left(\mathfrak{F}_l^* \sslash \text{SL}(N) \times \text{SL}(V_l)\right)$ is a principal bundle. Thus $\mathfrak{F}_l^* \sslash \text{SL}(N) \times \text{SL}(V_l)$ is a regular subvariety.  
	
\end{proof}

\begin{proof}[Proof of Theorem \eqref{mainthm3}]
		We have the morphism $\kappa_g:\overline{\mathfrak{U}}_{_{g, n, r }} \rightarrow \overline{M}_{_{g, n}}$. Let $[C_0]$ be a point in $\overline{M}_{_{g, n}}$ represented by a point $s \in \tt{S}_{_{g, n}}$ and we are writting $C_0$ for $\mathcal{X}_s$. Then the variety $Y_s$ is a subvariety of $\text{Hilb}^{p(t)}(C_0 \times \text{Gr}(V_l, r))$. We will have the flag variety $\left(\mathfrak{F}_l\right)_s \rightarrow Y_s$. Then the fiber of the morphism $\kappa_g$ at $[C_0]$ is $\left(\mathfrak{F}_l\right)_s \sslash \left(\text{Aut}(C_0, x^1, \cdots, x^n) \times \text{SL}(V_l)\right)$.\\
		
		We want to prove that $\left(\mathfrak{F}_l\right)_s$ has the singularity which is a product of analytic normal crossings. Since the forget morphism $\left(\mathfrak{F}_l\right)_s \rightarrow Y_s$ is smooth, it is enough to show that $ Y_s$ has the same type of singularity. Let $h \in Y_s$ be a point represented by a marked semistable curve $C$. Let $R=\text{Spec}(\hat{\mathcal{O}_{h, Y_s}})$ be a formal neighbourhood at $h$. Let $\Delta_{R} \hookrightarrow R \times (C_0 \times \text{Gr}(V_l, r))$ be the restriction of the universal curve to $R$. We have the induced morphism $\psi_{R}: \Delta_{R} \rightarrow R \times C_0$. Then by the versal property \eqref{versalproperty} there exists a morphism $R \rightarrow \mathcal{W}$ which will factor through the subvariety  $\mathcal{Q}$. Now it can be proved using methods similar to \eqref{Y is regular} that $R \rightarrow \mathcal{Q}$ is formally smooth.\\
		
	Let $\left(\mathfrak{F}_l\right)_s^{0} \hookrightarrow \left(\mathfrak{F}_l\right)_s$ be the open subvariety of $p_2$-stable sheaves. The action of $\text{SL}(V_l)$ on $\left(\mathfrak{F}_l\right)_s^{0}$ is free since the action on $(\text{F}_l)^0_s$ is free \eqref{action is free} and the morphism $\mathfrak{F}_l \rightarrow \text{F}_l$ is $\text{SL}(V_l)$ equivariant. Thus $\left(\mathfrak{F}_l\right)_s^{0}$ is a principal $\text{SL}(V_l)$ bundle over the geometric quotient $\left(\mathfrak{F}_l\right)_s^{0} \slash \text{SL}(V_l)$. Therefore the fiber of $\kappa_g$ at $[C_0]$ has the singularity which is a product of analytic normal crossing modulo the action of a finite group. In particular when the curve has no nontrivial automorphisms it is of the type of product of analytic normal crossing singularities.

	\end{proof}
	
	Recall that the open subvariety $\text{F}_l^s$ of $\text{F}_l$ is the locus of $p_2$-stable sheaves which has the natural induced isomorphism $\text{V}_l \cong \tt{H}^{0}(E)$. A parabolic sheaf $\mathcal{E}_*$ is $p_2$-stable if and only if for any non zero proper saturated subsheaf $\mathcal{F}$ of $\mathcal{E}$ with the induced parabolic structure on $\mathcal{F}$ we have either $par \mu(\mathcal{F}_*) < par \mu(\mathcal{E}_*)$ or $par \mu(\mathcal{F}_*) = par \mu(\mathcal{E}_*)$ and $\mu(\mathcal{F}) > \mu(\mathcal{E}) $.

	\begin{lemma}
		\label{$p_2$-stable simple}
		Every $p_2$-stable sheaf is simple.
	\end{lemma}
	\begin{proof}
		We will prove that for a $p_2$-stable sheaf $\text{Par-End}(\mathcal{E}_*) \cong \mathbb{C}$. Let $\phi:\mathcal{E}_* \rightarrow \mathcal{E}_*$ be a parabolic endomorphism. Let $p$ be a point of $C$. Then $\phi_p: \mathcal{E}_p \rightarrow \mathcal{E}_p$ is an automorphism of vector spaces over $\mathbb{C}$. Let $\lambda$ be an eigenvalue. We will get the parabolic endomorphism $\phi-\lambda \cdot I: \mathcal{E}_* \rightarrow \mathcal{E}_*$ which will be denoted by $\psi$. Let $\mathcal{F}$ be the kernel of $\psi$ which is a non zero subsheaf. If $\mathcal{F}=\mathcal{E}$ then we are done, so we will assume $\mathcal{F} \subsetneq \mathcal{E}$. The morphism $\psi$ will induce an isomorphism $\mathcal{E}/\mathcal{F} \cong \psi(\mathcal{E})$ where $\psi(\mathcal{E})$ is a subsheaf of $\mathcal{E}$. Since $\mathcal{F}$ is a saturated subsheaf, we will give the induced parabolic structure on $\mathcal{F}$ and $\psi(\mathcal{E})$. Since $\mathcal{E}_*$ is $p_2$-stable, by the above definition, let $par \mu(\mathcal{F}_*) < par \mu(\mathcal{E}_*)$, then $par \mu(\psi(\mathcal{E})_*) > par \mu (\mathcal{E}_*)$ which contradicts $p_2$-stability of $\mathcal{E}_*$. So let us assume $par \mu(\mathcal{F}_*) = par \mu(\mathcal{E}_*)$ and $\mu(\mathcal{F}) > \mu(\mathcal{E})$, which will again give a contradiction.  Thus we have $\mathcal{F}=\mathcal{E}$ which means $\phi=\lambda \cdot I$ for some $\lambda \in \mathbb{C}$.
	\end{proof}

	\begin{lemma}
		\label{action is free}
		The action of $\text{SL}(V_l)$ on $\text{F}_l^s$ is free.
	\end{lemma}
	
	\begin{proof}
		Note that the action of $\text{SL}(V_l)$ is fiber preserving. Let $(\text{F}_l)_s^0$ be the fiber of $\text{F}_l^s$ over a point $s \in \tt{S}_{_{g, n}}$ which is represented by the marked stable curve $\mathcal{X}_s=C$. The fiber $(\text{F}_l)_s^0$ is over the uniform rank $\tt{Q}uot$ scheme $\text{Q}uot^{r}(V_l \otimes \mathcal{O}_{C}, H)$ which is a subscheme of the $\text{Q}uot$ scheme $\text{Q}uot(V_l \otimes \mathcal{O}_{C}, H)$. We will prove that the stabilizer $\text{Stab}_{\text{SL}(V_l)}(\mathcal{E}_*)=\text{Par-Aut}(\mathcal{E_*})$, where $\mathcal{E}_*$ is an element of $\text{F}_l$. This combined with \eqref{$p_2$-stable simple} will give us the result.\\
		
		Let $g$ be an element of $\text{SL}(V_l)$ and $\mathcal{E}_* \in   \text{F}_l$. We will have the following diagram which depicts the action of $g$ on $\mathcal{E}_*$
		
		\begin{equation}
		\label{groupaction}
		\begin{tikzcd}[column sep=4em,row sep=4em]
		\text{V}_l \otimes \mathcal{O}_{C}  \arrow[r] \arrow[r,bend left,"k"] \arrow[rr,bend left,"k_{l_i+1}"] \arrow[rrr,bend left,"k_j"] \arrow[rrrr,bend left,"k_2"]&
		\mathcal{E}    \arrow[r] &
		\mathcal{E}/\mathcal{E}(-x^i) \arrow[r] &
		\cdots \arrow[r] &
		\mathcal{E}/F^i_2\mathcal{E}
		\\
		\text{V}_l \otimes \mathcal{O}_{C} \arrow[u, "g"] \arrow[ur, "k'"] \arrow[urr, "k_{l_i+1}'"] \arrow[urrr, "k_j'"] \arrow[urrrr, "k_2'"]& 
		\end{tikzcd}
		\end{equation}	
		
		If $g$ is a stabilizer of $\mathcal{E}_*$, then we will have $\text{ker}(k)=\text{ker}(k')$ and $\text{ker}(k_j)=\text{ker}(k_j')$ for $2 \le j \le (l_i+1)$. Then we will have the following commutative diagram
		
		\begin{equation}
		\begin{tikzcd}[column sep=4em,row sep=2em]
		\text{V}_l \otimes \mathcal{O}_{C}/\text{ker}(k) \ar[equal]{d}  \ar[r, "\sim", "k"']  & \mathcal{E} \ar[d, "\phi"] \\%
		\text{V}_l \otimes \mathcal{O}_{C}/\text{ker}(k') \ar[r, "\sim", "k'"' ] & \mathcal{E}
		\end{tikzcd}
		\end{equation} 
		
		which will induce an isomorphism $\phi:\mathcal{E} \rightarrow \mathcal{E}$. Using $\text{ker}(k_j)=\text{ker}(k_j')$ in the same way we will get an isomorphism $\phi_j: \mathcal{E}/F^i_j \mathcal{E}  \rightarrow \mathcal{E}/F^i_j \mathcal{E}$ which is compatible with $\phi$. Hence we will get a parabolic automorphism $\phi:\mathcal{E}_* \rightarrow \mathcal{E}_*$ in a natural way.\\
		
		Conversely, let $\phi: \mathcal{E}_* \rightarrow \mathcal{E}_*$ be a parabolic automorphism. We will have the quotient representation $k:\text{V}_l \otimes \mathcal{O}_{C} \rightarrow \mathcal{E}$ along with the induced isomorphism $\text{H}^{0}(k): \text{V}_l \rightarrow \text{H}^{0}(\mathcal{E})$. We will define the unique automorphism $g:\text{V}_l \rightarrow \text{V}_l$ so that the following diagram is commutative 
		
		\begin{equation}
		\begin{tikzcd}[column sep=4em,row sep=2em]
		\text{V}_l  \ar[d, "g"]  \ar[r, "\tt{H}^0(k)"]  & \tt{H}^{0}(\mathcal{E}) \ar[d, "\tt{H}^{0}(\phi)"] \\%
		\text{V}_l \ar[r, "\tt{H}^0(k)"] &  \tt{H}^{0}(\mathcal{E})
		\end{tikzcd}
		\end{equation} 
		
	We will prove that $g$ is a stabilizer of $\mathcal{E}_*$ in $\text{F}_l$. We will show that $\text{ker}(k)=\text{ker}(k')$ for the action of $g$ \eqref{groupaction}, similarly it will follow $\text{ker}(k_j)=\text{ker}(k_j')$.	From the definition of $g$ we will get the following commutative diagram
	
	\begin{equation}
	\begin{tikzcd}[column sep=4em,row sep=2em]
	\text{V}_l \otimes \mathcal{O}_{C} \ar[d, "g"]  \ar[r, "k"] \ar[dr, "k'"] & \mathcal{E} \ar[d, "\phi"] \\%
	\text{V}_l \otimes \mathcal{O}_{C} \ar[r, "k"] &  \mathcal{E}
	\end{tikzcd}
	\end{equation} 
		
	Since $\phi$ is an automorphism and $k'=\phi \circ k$, we have $\text{ker}(k) 	= \text{ker}(k')$. 	\\
	
	Both of these two morphisms are natural in the sense that there is no choice involved. From the definitions it follows that they are inverses of each other.  
	\end{proof}

	\appendix
	\section{moduli space of marked stable curves}
	\label{modulistablecurves}

\begin{definition}
	\label{stablecurve}	
\begin{enumerate}
	\item A {\em marked stable curve} is a marked prestable curve \eqref{markedsemistablecurve} which has finite automorphism {\em (preserving marked points)} group.\\

		\item {\em (combinatorial definition)}
		A  {\em marked prestable curve} is called a marked stable curve if every nonsingular rational component has at least $3$ special points {\em (either marked points or singular points)} and every nonsingular elliptic component has at least $1$ special points.
\end{enumerate}
	\end{definition}

\begin{definition}
	\begin{enumerate}
\item	A {\em family of marked stable curves} is a flat morphism $\mu:\tt{Y} \rightarrow T$ with sections $\sigma_i:\tt{T} \rightarrow \tt{Y}$ such that  ($\mu^{-1} (t), \sigma_1(t), \cdots, \sigma_n(t))$ is a marked stable curve for all $t \in T$. The family will be denoted by $(\mu: \tt{Y} \rightarrow \tt{T}, \sigma_1, \cdots , \sigma_n)$ \\

\item Two families $(\mu: \tt{Y} \rightarrow \tt{T}, \sigma_1, \cdots , \sigma_n)$ and $(\mu': \tt{Y'} \rightarrow \tt{T}, \sigma_1', \cdots , \sigma_n')$ are called equivalent if there exist an isomorphism $\phi: \tt{Y} \cong \tt{Y'}$ over $\tt{T}$ which is compatible with sections i.e., $\phi \circ \sigma_i=\sigma_i'$. The equivalance will be denoted by $(\mu: \tt{Y} \rightarrow \tt{T}, \sigma_1, \cdots , \sigma_n) \equiv (\mu': \tt{Y'} \rightarrow \tt{T}, \sigma_1', \cdots , \sigma_n')$
	\end{enumerate}
\end{definition}
	
	There exist a moduli space of isomorphism classes of marked stable curves of genus $g$ and $n$ marked points which is denoted by $\overline{M}_{_{g,n}}$. This moduli space is the Deligne-Mumford compactification of the moduli space $M_{_{g,n}}$ of isomorphism classes of smooth curves of genus $g$ and $n$ marked points. We will briefly mention the method of GIT construction of  $\overline{M}_{_{g,n}}$ \cite{MR2431236}. We will assume genus $g \ge 2$ for the rest of the section.
	
	\subsection{Brief construction of $\overline{M}_{_{g,n}}$}
	
	Let $(C, x^1, x^2, \cdots ,x^n)$ be a marked prestable curve. The dualizing sheaf $\omega_C$ is the sheaf of logarithmic $1$ form $f$ on the normalization $\tilde{C}$ which are regular except simple poles at $\{p_1^j, p_2^j:1 \le j \le c\}$ such that $\tt{Res}_{p_1^j}(f)+\tt{Res}_{p_2^j}(f)=0$ where $p_1^j, p_2^j$ are inverse image of the node $z^j$. Let $\mathcal
	{L}$ be the twisted line bundle $\omega_C(x^1+x^2+ \cdots+x^n)$.
	
	\begin{lemma}
		If  $(C, x^1, x^2, \cdots ,x^n)$ is a marked stable curve then $\mathcal{L}^{\otimes \rho}$ is very ample if $\rho \ge 3$. We also note that $H^1(C, \mathcal{L}^{\otimes \rho})=0$ for $\rho \ge 2$.
	\end{lemma}
	\begin{proof}
		This can be proved using arguments similar to \cite[Theorem 1.2]{MR262240} 
	\end{proof}
	
	We will consider the linear system corresponding to the very ample line bundle $\mathcal{L}^{\otimes \rho}$ with $\rho$ large enough e.g., $\rho \ge 5$ will be enough to embedd the curve $C$ in a projective space. \\
	
	Let the dimension of the linear system $\tt{dim}( H^0(C, \mathcal{L}^{\otimes \rho}))$ be $N$. The moduli problem is rigidified with the choice of an isomorphism $\tt{H}^0(C, \mathcal{L}^{\otimes \rho})) \cong \mathbb{C}^{N}$ and hence it induces an action of $\tt{SL}(N)$. Let $\tt{p}(t)$ be the Hilbert polynomial of the curve $C$ with respect to the line bundle $\mathcal{L}^{\otimes \rho}$ i.e., $\tt{p}(n)=\chi(\mathcal{L}^{\otimes (\rho \cdot n)})$. Let $\tt{Hilb}^{p(t)}(\P^{(N-1)})$ be the Hilbert scheme of curves in $\P^{N-1}$ of Hilbert polynomial $\tt{p}(t)$. Let the closed subscheme $\mathfrak{C} \hookrightarrow \tt{Hilb}^{p(t)}(\P^{(N-1)}) \times \P^{N-1}$ be the universal curve. We will consider the closed subscheme of incidence
	\begin{equation}
	\label{incidence}
		\tt{I} \hookrightarrow  Hilb^{p(t)} (\P^{N-1}) \times \underbrace{\P^{N-1} \times \cdots\times \P^{N-1}}_\text{n times}
	\end{equation}
	consisting of $(n+1)$ tuple \{$([C], x^1, x^2,  \cdots ,x^n):[C] \in \tt{Hilb}^{p(t)} (\P^{N-1})$ and $x^i \in C$  \}.\\
	
	We will have a natural forgetful map
	\begin{equation}
		\tt{Hilb}^{p(t)} (\P^{N-1}) \times \underbrace{\P^{N-1} \times \cdots\times \P^{N-1}}_\text{n times} \rightarrow Hilb^{p(t)} (\P^{N-1})
	\end{equation} 
	Taking the product with the universal curve $\mathfrak{C}$ we get the following closed subscheme
	\begin{equation}
		\mathfrak{C} \times \underbrace{\P^{N-1} \times \cdots\times \P^{N-1}}_\text{n times} \hookrightarrow \tt{Hilb}^{p(t)} (\P^{N-1}) \times \P^{N-1} \times \underbrace{\P^{N-1} \times \cdots\times \P^{N-1}}_\text{n times}
	\end{equation}
	We will take the following base change
	\begin{equation}
	\begin{tikzcd}[column sep=4em,row sep=2em]
	\mathfrak{C}_{\tt{I}} \arrow[r, hook]  \dar[hook]  & \mathfrak{C} \times (\P^{N-1})^n \dar[hook] \\%
	\tt{I} \times \P^{N-1} \rar[swap, hook] & \tt{Hilb}^{p(t)}(\P^{(N-1)}) \times \P^{N-1} \times (\P^{N-1})^n
	\end{tikzcd}
	\end{equation} 
	
	 The incidence subscheme $\tt{I}$ is called the Hilbert scheme of marked curves and the closed subscheme $\mathfrak{C}_{\tt{I}}$ is called the marked universal curve.
	We consider the locally closed subscheme $\tt{J} \hookrightarrow I$ whose points has the following property:\\
	
	$1$. The point $([C], x^1, x^2, \cdots ,x^n)$ has to be a  marked prestable curve \eqref{markedsemistablecurve}.\\
	
	$2$.  The embedding of the curve $C$ in $\P^{N-1}$ has to be non degenerate i.e., the image does not lie inside a hyperplane.\\
	
	$3$. There exist an isomorphism $\omega_C(x^1+x^2+ \cdots+x^n)^{\otimes \rho} \cong \mathcal{O}_{\P^{N-1}}(1)_{|C}$.\\
	
	Prestability and non degeneracy are open conditions. That the third property is a closed condition is verified in \cite[p. 58]{MR1492534}. Therefore we will get a locally closed subscheme $\tt{J}$ of $\tt{I}$. The elements of $\tt{J}$ parameterizes the marked stable curves. Restricting the marked universal curve we get the closed subscheme $\mathfrak{C}_{\tt{J}} \hookrightarrow \tt{J} \times \P^{N-1}$. The family $\mathfrak{C}_{\tt{J}} \rightarrow \tt{J}$ has natural sections $\sigma_i:\tt{J} \rightarrow \mathfrak{C}_{\tt{J}}$ which assigns the marked points $x^i$ to each point $([C], x^1, \cdots, x^n) \in \tt{J}$. This family has the local universal property for the moduli problem of $\overline{M}_{_{g, n}}$ in the following sense:
	\begin{proposition}
		\label{localuniproperty}
		Let $(\mu:\tt{Y} \rightarrow T, \sigma_1, \cdots, \sigma_n)$ be a family of marked stable curves. Then for all $t$ in $T$ there exist a neighbourhood $\tt{U}$ of $t$ and a morphism $\tt{U} \rightarrow \tt{J}$ such that $\tt{Y}|_{U} \equiv \mathfrak{C}_{\tt{J}} \times_{J} U$.
	\end{proposition}
	\begin{proof}
		For a proof we refer to \cite[Proposition 3.4]{MR2431236}.
	\end{proof}
	From the rigidification we have an action of $\tt{SL}(N)$ on $\P^{N-1}$. Thus $\tt{SL}(N)$ will induce a natural action on $\tt{Hilb}^{p(t)} (\P^{N-1})$ and so $\tt{SL}(N)$ will act diagonally on $\tt{I}$. It is clear that $\tt{J}$ and $\overline{\tt{J}}$ is invariant under the action of $\tt{SL(N)}$ with the closure of $\tt{J}$ being taken in the projective variety $\tt{I}$. The main GIT problem in \cite{MR2431236} is the study of the $\tt{SL}(N)$ action on $\tt{I}$.\\
	
	Let the vector space $\tt{W}=H^0(\P^{N-1}, \mathcal{O}_{\P^{N-1}}(m))$ and $\tt{W}_m=\wedge^{p(m)} W$. Then we will have a closed embedding $\tt{Hilb}^{p(t)}(\P^{N-1}) \hookrightarrow \P(W_m)$. This embedding is the composition of the Grothendieck embedding with the Pl\"{u}cker embedding. We choose the following ample line bundle 
	\begin{equation}
		\text{L}_{m, m'}=\left(\mathcal{O}_{\P(W_m)}(1) \boxtimes \left( \bigotimes_{i=1}^n \mathcal{O}_{\P^{N-1}}(m')\right)\right)_{\big | \tt{I}}
	\end{equation} 
	which gives a linearization for the $\tt{SL}(N)$ action on $\tt{I}$ \cite[pp. $17-18$]{MR2431236}. Now we will state the main result of \cite{MR2431236}.
	\begin{theorem}[\cite{MR2431236}, Theorem 6.3, Propositions 6.8, $\;$ 6.9 ]
		\label{GITtheoremmodcurves}
		There exist natural numbers $m, m'$( depends only on $g, n$) such that  $\overline{\tt{J}}^{ss}(\tt{L}_{m, m'})=\overline{\tt{J}}^{s}(\tt{L}_{m, m'})=\tt{J}$. 
			\end{theorem}

	 The GIT quotient ${\displaystyle \overline{\tt{J}} \sslash \tt{SL(N)}}$ is a projective variety. In fact the quotient is an orbit space which is denoted by  $\overline{M}_{_{g,n}}$. The moduli space $\overline{M}_{_{g,n}}$ is the coarse moduli space of the following functor.
	 \begin{definition}
	 	We define the moduli functor
	 	\begin{center}
	 		$\underline{\mathcal{M}}_{_{g,n}}$: $\tt{Sch}/\mathbb{C} \rightarrow \tt{Sets}$ 
	 	\end{center}
 	which assigns a scheme $\tt{T}$ to the set of equivalance classes of families of marked stable curves $(\mu:\tt{Y} \rightarrow \tt{T}, \sigma_1, \cdots, \sigma_n)$.
	 \end{definition}
	For the purpose of the paper we will identify the flat and projective family
	\begin{equation}
	\label{localunifamily}
		 \begin{tikzcd}
		\mathfrak{C}_{\tt{J}}  \arrow[d]   \\%
		\tt{J} \ar[u, bend left=50, "\sigma_i"]
		\end{tikzcd}
	\end{equation}
	 as $\mathcal{X} \rightarrow {\tt S_{_{g,n}}}$. We will have the relative dualizing sheaf ( a line bundle ) $\omega_{\mathcal{X}/{\tt S_{_{g,n}}}}$ and the relative very ample line bundle $\omega_{\mathcal{X}/{\tt S_{_{g,n}}}}(\sigma_1+\sigma_2+ \cdots +\sigma_n)$.

	\section{Universal moduli space of parabolic torsion free sheaves}
	\label{appendixb}

\begin{lemma}
	\label{p_2stabilitykeylemma}
	A parabolic sheaf $\mathcal{E}_*$ is $p_2$-stable($p_2$-semistable) if and only if for any non zero proper saturated subsheaf 
	$\mathcal{F}$ of $\mathcal{E}$ with the induced parabolic structure on $\mathcal{F}$ we have either $par \mu(\mathcal{F}_*)\footnote{This is a standard notation used for the expression in the LHS of \eqref{slopestability}} < par \mu(\mathcal{E}_*)$	or $par \mu(\mathcal{F}_*) = par \mu(\mathcal{E}_*)$ and $\mu(\mathcal{F}) > (\ge)\mu(\mathcal{E}) $.
\end{lemma}

\begin{proof}
	This easily follows from the definitions. For a proof we refer to \cite[ Remark 4.3.13]{schluter2011universal}.
\end{proof}
\begin{remark}
	\label{slopep_2stability}
We will have the following relations:

\{par $\mu$ stable sheaves\} $\subseteq$ \{par $p_2$ stable sheaves\} $\subseteq$ \{par $p_2$ semistable sheaves\} $\subseteq$ \{par $\mu$ semistable sheaves\}

The inclusions are in general strict.
\end{remark}

\begin{remark}
	As mentioned in the above remark \eqref{slopep_2stability} the notion of slope stability \eqref{slopestability} and $p_2$ stability \eqref{p_2stability} are not equivalent even in the case of curves. We will provide an example to show that the R.H.S inclusion is strict.\\
	
	Let us consider the rank $2$ vector bundle $E=\mathcal{O} \oplus \mathcal{O}(1)$ on $\P^1$. Let $P$ and $Q$ be the parabolic points on $\P^1$. At both $P$ and $Q$ we will define the following same parabolic structure with same weights on $E$
	
	\begin{equation}
	E \supseteq \mathcal{O} \oplus \mathcal{O} \supseteq \mathcal{O}(-1) \oplus \mathcal{O}
	\end{equation} 
	with attached weights
	\begin{equation}
		0 < \frac{1}{4} < \frac{3}{4} < 1
	\end{equation}  
	
	Then $\text{par} \hspace{0.1cm} \mu(E)=\frac{\text{deg}(E)+ \frac{1}{4} +\frac{3}{4}  +\frac{1}{4}+ \frac{3}{4}}{2}=\frac{3}{2}$. \\
	
	Let $\text{L}$ be a rank $1$ saturated subbundle of $E$. We consider the induced maps $\text{L} \rightarrow \mathcal{O}$ and $\text{L} \rightarrow \mathcal{O}(1)$. If either of them is zero then we will have either $\text{L} \cong \mathcal{O}$ or $\text{L} \cong \mathcal{O}(1)$, since $L$ is saturated. If both of them is nonzero, then either $L \cong \mathcal{O}$ or $\text{deg}(L) \le -1$. \\
	
	If $\text{deg}(L) \le -1$ then we can see that $\text{par} \hspace{0.1cm} \mu(L) < \text{par} \hspace{0.1cm} \mu(E)$. In both cases when $L \cong \mathcal{O}$ we will have the following induced parabolic strucrture 
	
	\begin{equation}
	\mathcal{O} \supseteq \mathcal{O} \supseteq \mathcal{O}(-1)
	\end{equation} 
	with attached weights
	\begin{equation}
	0 < \frac{1}{4} < \frac{3}{4} < 1
	\end{equation}  
	 
	 Then $\text{par} \hspace{0.1cm} \mu(\mathcal{O})=0+\frac{3}{4} +\frac{3}{4} =\frac{6}{4}= \text{par} \hspace{0.1cm} \mu(E)$ and $\mu(\mathcal{O}) < \mu(E)$. Then by \eqref{p_2stabilitykeylemma} $E$ can not be $p_2$-semistable.\\
	 
	 If $L \cong \mathcal{O}(1)$, then we will have the following induced parabolic structure 
	 
	 \begin{equation}
	 \mathcal{O}(1) \supseteq \mathcal{O} \supseteq \mathcal{O}
	 \end{equation} 
	 with attached weights
	 \begin{equation}
	 0 < \frac{1}{4} < \frac{3}{4} < 1
	 \end{equation} 
	 
	 Then we will get $\text{par} \hspace{0.1cm} \mu(\mathcal{O}(1))=1+\frac{1}{4} +\frac{1}{4} =\frac{6}{4}= \text{par} \hspace{0.1cm} \mu(E)$. Hence $E$ is $\mu$ semistable.
\end{remark}

\begin{lemma}[boundedness]
	\label{boundedness}
	Consider the local universal family $\mu:\mathcal{X} \rightarrow {\tt S_{_{g,n}}}$ \eqref{localunifamily}. Then the family of objects \bigg\{$p_2$ semistable sheaves with fixed numerical type $(r, d, r^i_j, \alpha^i_j)$ on $\mathcal{X}_s$:$s \in {\tt S_{_{g,n}}}$ \bigg\} is bounded.
\end{lemma} 

\begin{proof}
 For the proof we refer to \cite[Theorem 4.5.2]{schluter2011universal}.
\end{proof}

We have the natural closed embedding $\mathcal{X} \hookrightarrow {\tt S_{_{g,n}}} \times \P^{N-1}$. Thus for any $s \in S_{_{g,n}}$ we have $\mathcal{X}_s \hookrightarrow \P^{N-1}$. Therefore the sheaves in \eqref{boundedness} are all $p_2$ semistable sheaves of dimension $1$ in $\P^{N-1}$. So the notion of boundedness makes sense.
\begin{definition}[flat family of parabolic sheaves]
	\label{flatfamilyofparabolicsheaves}
	Let $T$ be a ${\tt S_{_{g,n}}}$ scheme. A flat family of parabolic sheaves $\mathcal{E}_*$ means a coherent sheaf $\mathcal{E}$ on $\mathcal{X}_T$ flat over $T$ with
	
	$\RNum{1}.$ A Parabolic filtration:
	\begin{equation}
	\label{relparfiltration}
		\mathcal{E}=F^i_1 \mathcal{E}\supset F^i_2 \mathcal{E} \supset F^i_3 \mathcal{E} \supset \cdots \supseteq F^i_{l_i+1} \mathcal{E}=\mathcal{E}(-D^i_T)
	\end{equation}
	such that $\mathcal{E}/F^i_j \mathcal{E}$ are flat over $T$ $\;\forall \;i, \; j$.
	
	$\RNum{2}.$ Along with attached weights:
	\begin{equation}
		0 \le \alpha^i_1 < \alpha^i_2 < \cdots <\alpha^i_{l_i} <1
	\end{equation}
	such that for $\;t \in T$ the restriction of $\mathcal{E}_*$ is a parabolic sheaf on $\mathcal{X}_t$ \eqref{qps} with respect to the divisors $D^i_t$.
\end{definition}
Let $G$ be a reductive group acting on two projective variety $X$ and $Y$. Let $\mathcal{L}, \mathcal{M}$ be two linearization of the $G$ action on $X$ and $Y$ respectively. Then $G$ will have a natural action on $X \times Y$. Let $p_1: X \times Y \rightarrow X$ and $p_2:X \times Y \rightarrow Y$ be the two projections.

\begin{lemma}
	There exist an integer $b \gg0$ such that the following holds:\\
	
	\label{GITlemma}
	$p_2^{-1}Y^{s}(\mathcal{M}) \subseteq (X \times Y)^{s}(p_1^* \mathcal{M}^{\otimes b} \otimes p_2^*\mathcal{L}) \subseteq (X \times Y)^{ss}(p_1^* \mathcal{M}^{\otimes b} \otimes p_2^*\mathcal{L}) \subseteq p_2^{-1}Y^{ss}(\mathcal{M})$
\end{lemma}

\begin{proof}
	This result has been mentioned in several places. This goes back to \cite[Proposition 5.1]{MR309940}. See also  $\;$\cite[ Proposition 2.18, Proposition 3.3.1]{MR1304906}, $\;$\cite[Proposition 5.1.1]{schluter2011universal}.
\end{proof}

\section{Some lemmas from Nagaraj-Seshadri}
	
\subsection{Gieseker functor}\label{Giefunctor} The Gieseker functor has been defined in \cite{MR739786}, \cite[Definition 7]{MR1687729}. The basic notations has been defined in the introduction of the subsection \eqref{reldivGie}. We recall $\mathcal{X}$ over $\tt{S}_{_{g, n}}$ is the local universal family for the moduli problem of $\overline{M}_{_{g,n}}$ \eqref{localunifamily}. Let $\mathcal{G}=\mathcal{G}(r, d)$ be the functor
	\begin{center}
			$\mathcal{G}: \text{Sch}/{\tt S_{_{g,n}}} \to \text{Sets}$ 
	\end{center}
\begin{center}
	$\mathcal{G}(T)= $  Set of closed subschemes $\Delta \hookrightarrow \mathcal{X}\times_S T \times_S \text{Gr}(V_l, r)$
\end{center}
such that it satisfies the following property: \\

$1.$ The induced projection morphism $\Delta \rightarrow T \times_S \text{Gr}(V_l, r)$ is also a closed immersion.\\

$2.$ The family $\Delta \rightarrow T$ is a flat family of curves. For a given $t \in T$ maps to $s \in {\tt S_{_{g,n}}}$ such that $\Delta_t$ is a marked semi-stable curve \eqref{markedGiecurve} and the induced morphism $\Delta_t \rightarrow \mathcal{X}_s$ is the collapsing map to its marked stable model. In addition the closed subscheme $\Delta_t \hookrightarrow \mathcal{X}_s \times_{\mathbb{C}} \text{Gr}(V_l, r)$ has Hilbert polynomial $\tt{p}(t)$. The Hilbert polynomial is defined with respect to the line bundle $\tt{det}$$(E_t)$ where $E_t$ is pullback of the tautological quotient bundle of $\text{Gr}(V_l, r)$.\\

$3.$ The vector bundle $E_t$ on $\Delta _t$ has rank $r$, degree $d$ such that $\text{dim}\ V_l=d+r(1-g)$.\\
The bundle $E_t$ has a natural quotient representation $V_l \otimes \mathcal{O}_{\Delta_t} \rightarrow E_t \rightarrow 0$ which is the pullback of the universal short exact sequence on $\text{Gr}(V_l, r)$. The induced map in cohomology has to be an isomorphism $V_l \cong \tt{H}^0(E_t)$. This implies $\tt{H}^1(E_t)=0$.

\begin{proposition}
	\label{boundednessproposition}
	Let $C$ be a marked semistable curve and $\tilde{C^j}$ be the closure $\overline{C \setminus R^j}$. Recall that $\tilde{C'}$ is the partial normalization $\overline{C \setminus {\cup R^j}}$ of the marked stable curve $C'$. Let $E$ be a strictly positive vector bundle \eqref{markedGiecurve} on the curve $C$ which satisfies the following conditions\\
	
	$1.$ We have the cohomology vanishing $\tt{H}^1(\tilde{C^j}, I_{p_1^j, p_2^j} E_{|\tilde{C^j}})=0$. In particular this implies that $\tt{H}^0(\tilde{C^j}, E|_{\tilde{C^j}}) \rightarrow E_{p_1^j} \oplus E_{p_2^j}$ is surjective.\\
	$2.$ The canonical map $ \tt{H}^0(\tilde{C^j}, I_{p_1^j, p_2^j} E|_{\tilde{C^j}}) \rightarrow I_{p_1^j, p_2^j} E|_{\tilde{C^j}}/I^2_{p_1^j, p_2^j} E|_{\tilde{C^j}}$ is surjective.\\
	$3.$ The canonical map $ \tt{H}^0(\tilde{C'}, I_Z E|_{\tilde{C'}}) \rightarrow  E|_{\tilde{C'}}/I_x^2 E|_{\tilde{C'}}$ is surjective for $x \in \tilde{C'} \setminus Z$, where $Z=\cup\{p_1^j, p_2^j\}$.\\
	$4.$ The canonical map $ \tt{H}^0(\tilde{C'}, I_Z E|_{\tilde{C'}}) \rightarrow  E_{s_1} \oplus E_{s_2}$  is surjective for all pair of points $s_1, s_2 \in \tilde{C'} \setminus Z$ such that $s_1 \ne s_2.$\\
	
	Then we will have $\tt{H}^1(C, E)=0$. The global sections $\tt{H}^0(E)$ generates $E$ and the natural morphism $C \rightarrow \tt{Gr}(H^0(C, E), r)$ is a closed immersion.
	
\end{proposition}
\begin{proof}
	This can be proved using similar methods of \cite[Proposition 4]{MR1687729}.
\end{proof}

\begin{lemma}
	\label{NSlemma}
	Suppose we have the following commutative diagram
	\begin{equation}
	\begin{tikzcd}
	Z \arrow[rr, "\pi"] \arrow[dr, "p"]
	& & W \arrow[dl, "q"]\\%
	& T
	\end{tikzcd}
	\end{equation}
	such  that $p$  and  $q$  are projective  morphisms  (which  implies $\pi$  is proper), $\pi_* \mathcal{O}_Z = \mathcal{O}_W$ 
	and $p$ is flat.  Let $E$  be  a  vector  bundle  on $Z$  such  that $R^i (\pi_t)_* E_t=0$ for $i \ge 1$, then the higher direct images of $E$ behaves well with respect to restriction i.e., 
	$(R^i\pi_*E)_t \cong R^i(\pi_t)_* E_t$ for $i \ge 0$.
\end{lemma}

\begin{proof}
	For the proof of $i=0$ case we refer to \cite[Lemma 4]{MR1687729}. Using similar argument with little modification one can prove that the sheaf $R^i\pi_*E=0$ for $i >0$. Here we will briefly mention the arguments for the case of $i >0$.\\
	
	Let $\mathcal{O}_{W}(1)$ be a relative ample line bundle. Let $\mathcal{O}_{Z}[1]$ be the pullback $\pi^*\mathcal{O}_{W}(1)$. Let $E[n]$ be the tensor product $E \otimes \mathcal{O}_{Z}[n]$. Without loss of generality we can assume $T=\text{spec}(A)$. Since $q$ is projective by Serre correspondence $R^i\pi_*E$ is the sheaf associated to the graded module $\oplus_{n \ge 0} \text{H}^0(q_*(R^i\pi_*E(n)))$.\\
	
	We will use the Grothendieck spectral sequence for composite of two functors $E_2^{i,j}=R^i q_*(R^j \pi_* (E[n])) \implies R^{i+j} p_*(E[n])$. From the spectral sequence it follows that $R^i p_*(E[n]) \cong q_*(R^i \pi_*(E[n]))$ for all $i \ge 0$ and for sufficiently large $n$. Therefore we have 
	\begin{equation}
	\label{spectral sequence}
		\oplus_{n \gg 0} \text{H}^0(q_*(R^i\pi_*E(n))) \cong \oplus_{n \gg 0} \text{H}^0(R^i p_*(E[n]))
	\end{equation}
	
	Using the projection formula We will have 
	\begin{equation}
		R^i (\pi_t)_* (E_t[n]) \cong (R^i (\pi_t)_* E_t)(n) =0
	\end{equation}
	 for $i>0$ and for $n \ge 0$. This implies $H^i(Z_t, E_t[n]) \cong H^i(W_t, ((\pi_t)_*E_t)(n))$ for all $i \ge 0$. Since $\mathcal{O}_{W}(n)$ is relatively ample, there exist $n_0$ such that for $i>0$, $H^i(W_t, ((\pi_t)_*E_t)(n))=0$ for $n \ge n_0$ and for all $t$. This in turn implies that $H^i(Z_t, E_t[n])=0$ for $n \ge n_0$ and $i >0$. From the semicontinuity theorem of Grauert it follows $R^i p_*(E[n])=0$ for $i>0$ and $n \ge n_0$. From \eqref{spectral sequence} it follows that graded module associated to the sheaf $R^i\pi_*E$ vanishes after finitely many terms of the grading. Therefore the associated sheaf $R^i\pi_*E=0$.

\end{proof}
	
	\bibliographystyle{plainurl}
	\bibliography{universal_moduli_Revised_final_draft_20-05-22.bib}

\end{document}